\documentclass[a4paper, 10pt, twoside, notitlepage]{amsart}

\usepackage{amsmath,amscd}
\usepackage{amssymb}
\usepackage{amsthm}
\usepackage{comment}
\usepackage{graphicx, xcolor}

\usepackage{mathrsfs}
\usepackage[ocgcolorlinks, linkcolor=blue]{hyperref}

\usepackage[ocgcolorlinks,linkcolor=blue]{hyperref}

\newcommand{\LC}{\left(}
\newcommand{\RC}{\right)}

\theoremstyle{plain}
\newtheorem{thm}{Theorem}[section]
\newtheorem{prop}{Proposition}[section]
\newtheorem{lem}[prop]{Lemma}
\newtheorem{cor}[prop]{Corollary}

\newtheorem{defi}[prop]{Definition}
\newtheorem{rmk}[prop]{Remark}

\numberwithin{equation}{section}
\newcommand {\R} {\mathbb{R}} 
 \newcommand {\N} {\mathbb{N}}
 
\newcommand {\p} {\partial}

\newcommand{\eps}{\epsilon}
\newcommand{\vareps}{\varepsilon}

\newcommand {\supp} {\text{supp}}

\newcommand{\wt}{\widetilde}

\newcommand{\Q}{\mathcal{Q}}

\newcommand{\norm}[1]{\lVert #1 \rVert}         % Formatting for the norm
\DeclareMathOperator{\F} {\mathcal{F}}

\title[]{Monotonicity-based inversion of fractional semilinear elliptic equations with power type nonlinearities}

\author[Yi-Hsuan Lin]{Yi-Hsuan Lin}
\address{Department of Applied Mathematics, National Chiao Tung University, Taiwan}
\email{yihsuanlin3@gmail.com}

\begin{document}
	
	\maketitle
	
	\begin{abstract}
		We investigate the monotonicity method for fractional semilinear elliptic equations with power type nonlinearities. We prove that if-and-only-if monotonicity relations between coefficients and the derivative of the Dirichlet-to-Neumann map hold. Based on the strong monotonicity relations, we study a constructive global uniqueness for coefficients and inclusion detection for the fractional Calder\'on type inverse problem. Meanwhile, we can also derive the Lipschitz stability with finitely many measurements. The results hold for any $n\geq 1$.
		
		\medskip
		
		
		%\noindent{\bf Mathematics Subject Classification (2010)}: 
		
	\end{abstract}

	\tableofcontents

	\section{Introduction}\label{Sec 1}
	In this work, we extend the monotonicity method \cite{harrach2017nonlocal-monotonicity,harrach2020monotonicity} to the case of fractional semilinear elliptic equations with power type nonlinearities. The mathematical formulation is given as follows.
	Let $\Omega \subset \R^n $ be a bounded domain with $C^{1,1}$-boundary $\p \Omega$, for $n\geq 1$. For $0<s<1$, and any $m\geq 2$, $m\in \N$. Let $q\in L^\infty(\Omega)$ be a potential, then we consider the Dirichlet problem for the fractional semilinear elliptic equation with power type nonlinearities 
	\begin{align}\label{Main equation}
		\begin{cases}
		(-\Delta)^s u + qu^m =0 & \text{ in }\Omega ,\\
		u=f &\text{ in } \Omega_e :=\R^n \setminus \overline{\Omega}.
		\end{cases}
	\end{align}
	The well-posedness of \eqref{Main equation} holds for any sufficiently small exterior data $f$ in an appropriate function space, which will be demonstrated in Section \ref{Sec 2}.
	Here the fractional Laplacian $(-\Delta)^s $ is defined via the integral representation
	\begin{align}\label{fractional Laplacian}
	(-\Delta)^{s}u=c_{n,s}\mathrm{P.V.}\int_{\mathbb{R}^{n}}\dfrac{u(x)-u(y)}{|x-y|^{n+2s}}dy,
	\end{align}
	for $u\in H^s(\mathbb R^n)$, where P.V. denotes the principal value and  
	\begin{equation}
	c_{n,s}=\frac{\Gamma(\frac{n}{2}+s)}{|\Gamma(-s)|}\frac{4^{s}}{\pi^{n/2}}\label{c(n,s) constant}
	\end{equation}
	is a constant that was explicitly calculated in \cite{di2012hitchhiks}. Here $H^s(\R^n)$ is the fractional Sobolev space, which will be introduced in Section \ref{Sec 2}. On one hand, we want to emphasize that the regularity condition of $\p \Omega\in C^{1,1}$ is needed due to the well-posedness and suitable $L^p$-estimates for the fractional Laplacian. On the other hand, it is worth mentioning that in the study of fractional inverse problems for linear equations, in general we do not need regularity assumption on the domain.

    In this article, we study the fractional type Calder\'on problem for the equation \eqref{Main equation} of reconstruction an unknown potential $q$ from the (exterior) \emph{Dirichlet-to-Neumann} (DN) map $\Lambda_q$:
	\begin{align*}
		\Lambda_q: H^s(\Omega_e)\to H^s(\Omega_e)^\ast, \quad f\mapsto (-\Delta)^s u|_{\Omega_e},
	\end{align*}
	for any sufficiently small exterior data $f\in H^s(\Omega_e)$, where $u\in H^s(\R^n)$ is the solution of \eqref{Main equation} and $H^s(\Omega_e)^\ast$ is the dual space of $H^s(\Omega_e)$. For the sake of self-containedness, we will provide the proof of the well-posedness of \eqref{Main equation} under the smallness condition of exterior data, which implies that the DN map $\Lambda_q$ of \eqref{Main equation} is well-defined in the exterior domain $\Omega_e$. The fractional Calder\'on problem was first proposed by Ghosh-Salo-Uhlmann \cite{ghosh2016calder}, and related fractional inverse problems have been investigated by many researchers recently, such as  \cite{CLL2017simultaneously,cekic2020calderon,ghosh2017calder,GRSU18,harrach2017nonlocal-monotonicity,harrach2020monotonicity,LL2020inverse,lai2019global,RS17,LLR2019calder} and the references therein. The key ingredients in the fractional inverse problems are the \emph{strong uniqueness} (Proposition \ref{Prop: strong uniqueness}) and the \emph{Runge approximation} (Theorem \ref{Thm:runge}) in  $L^p(\Omega)$, for $p>1$. Based on these properties, many researchers have developed the fractional Calder\'on problem with partial data, monotonicity-based inversion formula and simultaneously recovering problems.

	The research of fractional semilinear Schr\"odinger equations arises in the quantum effects in Bose-Einstein Condensation \cite{Uzar}. In the ideal boson systems, the Gross-Pitaevskii equations characterizes condensation of weakly interacting boson atoms at a low temperature, wherever the probability density of quantum particles is conserved. Moreover, in the inhomogeneous media with long-range or nonlocal interactions between particles, this yields the density profile no longer retains its shape as in the classical Gross-Pitaevskii equations. This dynamics can be described by the fractional Gross-Pitaevskii equation, regarded as the fractional semilinear Schr\"odinger equation, in which the turbulence and decoherence emerge. It was investigates in \cite{Kay} that the turbulence appears from the nonlocal property of the fractional Laplacian; while the local nonlinearity helps maintain coherence of the density profile.  
	
	In general, it is known that the nonlinear and nonlocal problems are harder than their local counterparts for forward mathematical problems.
For the local case, i.e., $s=1$, one can consider the analogous inverse boundary value problem for the semilinear elliptic equation $\Delta u + a(x,u)=0$ in $\Omega$ with $u=f$ on $\p \Omega$. Similar inverse problems are recently treated in the independent works \cite{FO19,LLLS2019nonlinear}. By using the knowledge of the corresponding DN map, the authors \cite{FO19,LLLS2019nonlinear} have introduced the \emph{higher order linearization} method, to investigate that the unknown coefficients can be uniquely determined by its associated DN map (on the boundary). In addition, \cite{LLLS2019partial,KU2019remark,KU2019partial} have extended the unique determination results into the partial data setup, and the key ingredient is also relied on the higher order linearization. More specifically, the \emph{first} linearization will make the unknown coefficients disappear, so that one can apply the \emph{density} property of the scalar products of \emph{harmonic functions} (see \cite{calderon,ferreira2009linearized}). For general linear elliptic equations, one needs more complicated results to prove the density property of the scalar products of solutions to the certain equation, which might involve the \emph{complex geometrical optics} solutions.

	Very recently, Lai and myself \cite{LL2020inverse} have studied related inverse problems for fractional semilinear elliptic equations. We can recover the unknown coefficients and obstacles by using the higher order linearization, where we have simply used a single parameter $\eps$. In fact, we only need to utilize a single exterior measurement to recover coefficient and obstacle simultaneously. The goal of this work is to study related fractional inverse problems for \eqref{Main equation} via the \emph{monotonicity tests}. 
	
	Let us formulate the \emph{if-and-only-if} monotonicity relations in the following. For any potentials $q_1,q_2 \in L^\infty(\Omega)$, we will use the monotonicity arguments and localized potentials for the linearized equations to show that 
	\begin{align}\label{if-and-only-if monotonicity}
		q_1 \leq q_2 \quad \text{ if and only if }\quad (D^m\Lambda_{q_1})_0\leq (D^m\Lambda_{q_2})_0,
	\end{align}
	where $m\in \N$ is the integer of the fractional elliptic equation \eqref{Main equation} with power type nonlinearities $q_j(x)u^m$, for $j=1,2$. Here $q_1\leq q_2$ means that $q_1(x)\leq q_2(x)$ for almost everywhere (a.e.) $x\in \Omega$.

	In this work, the inequality $(D^m\Lambda_{q_1})_0\leq (D^m\Lambda_{q_2})_0$ in \eqref{if-and-only-if monotonicity} is denoted in the sense that 
	\begin{align}\label{suitable sense of monotonicity}
     \langle \underbrace{[(D^m \Lambda_{q_1})_0-(D^m\Lambda_{q_2})_0]}_{m\text{-linear form}}\underbrace{(g,\ldots ,g)}_{m-\text{vector}} , h \rangle \leq 0,
	\end{align}
	for any $g\in C^\infty_c(\Omega_e)$ and for some suitable $h\in C^\infty_c(\Omega_e)$. The exterior data $h\in C^\infty_c(\Omega_e)$ would be chosen differently when the integer number $m$ is even or $m$ is odd. We will give more detailed discussions in Section \ref{Sec 3}. Here $(D^m \Lambda_q)_0$ denotes the $m$-th order derivative of the DN map $\Lambda_q$ evaluated at the $0$ exterior data, and it can be computed directly from 
	\[
	\left.(D^m\Lambda_q)_0(g,\ldots, g)\right|_{\Omega_e}=\left. \left. \p^m_\eps \right|_{\eps=0}(-\Delta)^s u_{\eps g}\right|_{\Omega_e},
	\]
	where $u_{\eps g}\in H^s(\R^n)$ is the solution of 
	$$
	(-\Delta)^s u + qu^m=0 \text{ in } \Omega \quad \text{ with } \quad u=\eps g \text{ in } \Omega_e.
	$$ 
	We will characterize the preceding discussions in Section \ref{Sec 2} with more details. In addition, once we know the information of exterior measurements $\Lambda_q$, then we can determine the $m$-th order derivative of $\Lambda_q$.

	The first main result in this work is that the if-and-only-if monotonicity relations \eqref{if-and-only-if monotonicity} yield a constructive uniqueness proof of the potential $q(x)$ by knowing the knowledge of the $m$-th order derivative of the DN map $\Lambda_q$. The first main result in this paper is stated as follows.
	
	\begin{thm}[The if-and-only-if monotonicity relations]\label{Thm: If-and-only-if monotonicity}
		Consider $\Omega\subset \R^n$, $n\geq 1$ to be a bounded domain with $C^{1,1}$ boundary $\p\Omega$, and $0<s<1$.  Let $q_1, q_2 \in L^\infty(\Omega)$, $m\geq 2$ and $m\in \N$.
		Let $\Lambda_{q_j}$ be the DN maps of the semilinear elliptic equations $(-\Delta)^s u + q_j u^{m}=0$ in $\Omega$ for $j=1,2$.
		Then we have 
		\begin{align}\label{if and only if monotonicity in Sec 4}
		q_1 \geq q_2 \text{ a.e. in }\Omega \quad \text{ if and only if } \quad (D^m_0 \Lambda_{q_1})_0 \geq (D^m\Lambda_{q_2})_0.
		\end{align}
	\end{thm}

\begin{rmk} It is worth mentioning that 
\begin{itemize}
	\item[(a)] 	When $s=1$, i.e., for the local case, one can only expect that the monotonicity relations of potentials will imply the monotonicity relations of the corresponding DN maps. It is hard to show a converse statement to be true of the monotonicity formula. Fortunately, with the aids of the strong uniqueness of the fractional Laplacian, we are able to prove Theorem \ref{Thm: If-and-only-if monotonicity}(see Section \ref{Sec 4}), which is similar to the works \cite{harrach2017nonlocal-monotonicity,harrach2020monotonicity}.
	
	\item[(b)] It is natural to consider the $m$-th order derivative of DN map $(D^m\Lambda_q)_0$ instead of the original DN map $\Lambda_q$. Due to the well-posedness, one can trace the information of $(D^k \Lambda_q)_0$ for all $k\in \N$, and one cannot see any differences of $(D^k \Lambda_q)_0$ for any $k=0,1,\cdots, m-1$ (see Section \ref{Sec 2}).
\end{itemize}
\end{rmk}

The proof of Theorem \ref{Thm: If-and-only-if monotonicity} is based on the monotonicity formulas and the localized potentials for the fractional Laplacian (see Section \ref{Sec 3} and Section \ref{Sec 4}). 
Thanks to the strong uniqueness for the fractional Laplacian, one can approximate any $L^a$ function by solutions of the fractional Laplacian, for any $a >1 $. Then one can construct the localized potentials for the fractional Laplacian by using the standard normalization technique. 

In the study of inverse boundary value problem, the technique of combining monotonicity relations with localized potentials \cite{gebauer2008localized} is a useful approach, and this method has already been studied extensively in a number of results, such as  \cite{arnold2013unique,barth2017detecting,brander2018monotonicity,griesmaier2018monotonicity,harrach2009uniqueness,harrach2012simultaneous,harrach2017nonlocal-monotonicity,harrach2018localizing,HPS2019dimension,HPSmonotonicity,harrach2010exact,harrach2013monotonicity,harrach2017local,seo2018learning}. Also, several works have built practical reconstruction based on monotonicity properties \cite{daimon2020monotonicity,garde2017comparison,garde2019reconstruction,garde2017convergence,garde2019regularized,harrach2015combining,harrach2016enhancing,harrach2018monotonicity,harrach2015resolution,maffucci2016novel,su2017monotonicity,tamburrino2002new,tamburrino2016monotonicity,ventre2017design,zhou2018monotonicity}.

The second main result is that we investigate the inverse obstacle problem for the exterior problem \eqref{Main equation} of determining regions where a potential $q\in L^\infty(\Omega)$ changes from a known potential $q_0 \in L^\infty(\Omega)$. Our goal is not only to show the unique determination the unknown obstacle from the exterior measurements, but also we will give a reconstruction formula of the support $q-q_0$ by comparing $(D^m\Lambda_q)_0$ with $(D^m\Lambda_{q_0})_0$.
The potential $q_0$ is denoted as a background coefficient, and $q$ denotes the coefficient function in the presence of anomalies or scatterers.

In the spirit of \cite{garde2017comparison,harrach2013monotonicity,harrach2017nonlocal-monotonicity,harrach2020monotonicity}, we will show that the support of $q-q_0$ can be reconstructed via the monotonicity tests. 
Let $M\subset \Omega$ be a measurable set, and we define the \emph{testing operator} $T_M:H^s(\Omega_e)^m\to H^s(\Omega)^\ast $ via the pairing that 
\begin{align}\label{testing operator}
	\langle (T_M )(g,\ldots, g),  h \rangle :=	\int_{\Omega_e} (T_M )(g,\ldots, g) h\, dx =\int_M v_g^m v_h \, dx,
\end{align}
where $T_M$ is regarded as an $m$-form acting on $m$-vector valued functions. Here $v_g$ and $v_h$ are the solution of the fractional Laplacian in $\Omega$ with $v_g=g$ and $v_h=h$ in $\Omega_e$, respectively. Notice that the testing operator \eqref{testing operator} can be computed since we know the location of the measurable set $M\subset \Omega$ and the information of $v_g,v_h$ once $g,h\in C^\infty_c(\Omega_e)$ are given in the exterior domain $\Omega_e$.

The following theorem shows that the support of $q-q_0$ can be found by shrinking closed set. We also refer readers to \cite{harrach2013monotonicity,garde2019regularized,harrach2017nonlocal-monotonicity,harrach2020monotonicity} for the linear cases.
\begin{thm}[Unknown inclusion detection]\label{Thm:support_from_closed_sets}
	Consider $\Omega\subset \R^n$, $n\geq 1$ to be a bounded domain with $C^{1,1}$ boundary $\p\Omega$, and $0<s<1$.  Let $q_1, q_2 \in L^\infty(\Omega)$, $m\geq 2$ and $m\in \N$.
	Let $\Lambda_{q_j}$ be the DN maps of the semilinear elliptic equations $(-\Delta)^s u + q_j u^{m}=0$ in $\Omega$ for $j=1,2$.
	For each closed subset $C\subseteq \Omega$, 
	\begin{align*}
		\lefteqn{\mathrm{supp}(q-q_0)\subseteq C,}\\
		& \quad \text{ if and only if } \quad  \exists \alpha>0:\ 
		-\alpha {T}_C \leq (D^m\Lambda_q)_0-(D^m\Lambda_{q_0})_0\leq   \alpha {T}_C.
	\end{align*}
	Thus,
	\begin{align*}
		\lefteqn{\mathrm{supp}(q-q_0)}\\
		&=\bigcap \left\{ C\subseteq \Omega \text{ closed}:\ \exists \alpha>0:\ 	-\alpha {T}_C \leq (D^m\Lambda_q)_0-(D^m\Lambda_{q_0})_0\leq   \alpha {T}_C\right\}.
	\end{align*}
\end{thm}

Note that Theorem \ref{Thm:support_from_closed_sets} is not a deterministic result, but it is a reconstruction result. The proof of Theorem \ref{Thm:support_from_closed_sets} can be regarded as an application of Theorem \ref{Thm: If-and-only-if monotonicity}. Via the monotonicity tests, we can give a reconstruction algorithm by utilizing the testing operator $T_M$ in Section \ref{Sec 4}.

The last main contribution of this article is about the \emph{Lipschitz stability} of the fractional inverse problem with finitely many measurements. The Lipschitz stability with finitely many measurements has been studied by in various mathematical settings, we refer the reader to \cite{HM2019global_stability,harrach2020monotonicity,sincich2007lipschitz,RS2018lipschitz} and references therein for more detailed descriptions. In this work, we only consider the case that the set $\Q \subset L^\infty(\Omega)$ is a finite-dimensional subspace of piecewise analytic functions, and 
\[
\Q_{\lambda}:=\left\{q\in \Q: \ \norm{q}_{L^\infty(\Omega)}\leq \lambda \right\},
\]
for some constant $\lambda>0$.

\begin{thm}\label{Thm:baby stability}
	Consider $\Omega\subset \R^n$, $n\geq 1$ to be a bounded domain with $C^{1,1}$ boundary $\p\Omega$, and $0<s<1$.  Let $q_1, q_2 \in L^\infty(\Omega)$, $m\geq 2$ and $m\in \N$.
	Let $\Lambda_{q_j}$ be the DN maps of the semilinear elliptic equations $(-\Delta)^s u + q_j u^{m}=0$ in $\Omega$ for $j=1,2$.
	Then there exists a constant $c_0>0$ such that 
	\begin{align}\label{stability estimate}
	 \norm{(D^m\Lambda_{q_1})_0 - (D^m\Lambda_{q_2})_0}_{\ast}\geq c_0 \norm{q_1-q_2}_{L^\infty(\Omega)}, 
	\end{align}
	for any $q_1, q_2 \in \Q_{\lambda}$.
\end{thm}
 
 \begin{rmk} The operator norm  $\norm{\cdot}_{\ast}$ for the $m$-th order derivative DN map is defined by 
 	\begin{align*}
 	\norm{A}_{\ast}
 	=\sup \left\{\left|\langle A(g,\ldots, g), h \rangle\right|: \ g,h\in  C^\infty_c(\Omega_e), \norm{g}_{H^s}=\norm{h}_{H^s}=1 \right\}.
 	\end{align*}
 \end{rmk}

One can show that a sufficiently high number
of the exterior DN maps uniquely determines a potential in $\Q_\lambda$ and prove a Lipschitz stability result for the equation \eqref{Main equation}. In order to formulate the result, let us denote the orthogonal projection operators from $H^s(\Omega_e)$ to a subspace $H$ by $P_H$,
i.e. $P_H$ is the linear operator with
\[
P_H:\ H^s(\Omega_e)\to H, \quad P_H g:=\begin{cases}
g & \text{ if $g\in H$,}\\ 0 & \text{ if $g\in H^\perp\subseteq H^s(\Omega_e).$}
\end{cases}
\]
$P_H':\ H^*\to H^s(\Omega_e)^*$ stands for the dual operator of $P_H$. 

\begin{thm}\label{Thm:stability}
	Let $\Omega\subset \R^n$, $n\geq 1$ be a bounded domain with $C^{1,1}$ boundary $\p\Omega$, and $0<s<1$.  Let $m\geq 2$, $m\in \N$.
	Let $q_1, q_2 \in L^\infty(\Omega)$, 
	and $\Lambda_{q_j}$ be the DN maps of the semilinear elliptic equations $(-\Delta)^s u + q_j u^{m}=0$ in $\Omega$ for $j=1,2$.
	For every sequence of subspaces 
	\[
	H_1\subseteq H_2\subseteq H_3\subseteq ...\subseteq H^s(\Omega_e), \quad \text{ and } \quad
	\overline{\bigcup_{\ell \in \N} H_\ell}=H^s(\Omega_e),
	\]
	there exists $k\in \N$, and $c>0$, so that
	\begin{align}\label{Lipschitz stability estimate}
	\left\| P_{H_\ell}'  \left( (D^m\Lambda_{q_2})_0-(D^m\Lambda_{q_1})_0 \right)P_{H_\ell}\right\|_{\ast}
	\geq c \norm{q_2-q_1}_{L^\infty(\Omega)}
	\end{align}
	for all $q_1,q_2\in \Q_{\lambda}$ and all $l\geq k$.
\end{thm}

The article is structured as follows. In Section \ref{Sec 2}, we offer preliminary results for function space (fractional Sobolev spaces and H\"older spaces). We also proved the well-posedness of \eqref{Main equation}, i.e., there exists a unique solution $u$ of \eqref{Main equation}, whenever the exterior Dirichlet data $f$ are sufficiently small. In Section \ref{Sec 3}, we derive the monotonicity relations between potentials and its corresponding $m$-th order derivative DN maps. 
By combining with the monotonicity relations and localized potentials, we can prove the converse monotonicity relations in Section \ref{Sec 4}, so that we can prove our main results Theorem \ref{Thm: If-and-only-if monotonicity} and Theorem \ref{Thm:support_from_closed_sets}. We prove the Lipschitz stability results in Section \ref{Sec 5}. Finally, we recall some known results that the $L^p$-type estimates of solutions, and the maximum principle of the fractional Laplacian in Appendix \ref{Appendix} and Appendix \ref{Appendix max}, respectively.

\section{Preliminaries}\label{Sec 2}

In this section, we introduce function spaces and well-posedness of the Dirichlet problem \eqref{Main equation}. The well-posedness of $(-\Delta)^s u +a(x,u)=0$ has been proved in \cite{LL2020inverse}. We give a similar proof for the sake of completeness under a slightly weaker regularity assumption on the coefficient $a(x,u)=q(x)u^{m}$, when $q\in L^\infty(\Omega)$ and for $m\geq 2$, $m\in \N$.
Let us recall several function spaces which we will use in the rest of the paper.

\subsection{Function spaces}

Recalling the definition H\"older spaces as follows. Let $D\subset\R^n$ be an open set, $k\in \N \cup \{0\}$ and $0<\alpha <1$, then the space $C^{k,\alpha}(D)$ is defined by
\[
C^{k,\alpha}(D):=\left\{f:D\to \R:\ \norm{f}_{C^{k,\alpha}(D)}<\infty \right\}.
\]
The norm $\norm{\cdot}_{C^{k,\alpha}(D)}$ is given by 
\begin{align*}
	\norm{f}_{C^{k,\alpha}(D)}:=&\sum_{|\beta|\leq k}\norm{\p ^\beta f}_{L^\infty(D)}+\sum_{|\beta|=k}\sup_{\begin{subarray}{c}
		x\neq y, \\
		x,y\in D
		\end{subarray}}\frac{|\p ^\beta f(x)-\p^\beta f(y)|}{|x-y|^\alpha} \\
	=&\sum_{|\beta|\leq k}\norm{\p ^\beta f}_{L^\infty(D)} +\sum_{|\beta|=k}[\p ^\beta f]_{C^\alpha(D)}
\end{align*}
where $\beta=(\beta_1,\ldots,\beta_n)$ is a multi-index with $\beta_i \in \N \cup \{0\}$ and $|\beta|=\beta_1 +\ldots +\beta_n$. Here $[\p ^\beta f]_{C^\alpha(D)}$ is denoted as the seminorm of $C^{0,\alpha}(D)$.
Furthermore, we also denote the space
\[
C_0^{k,\alpha}(D):=\text{closure of }C^\infty_c(D) \text{ in }C^{k,\alpha}(D).
\]
We also denote $C^\alpha(D) \equiv C^{0,\alpha}(D)$ when $k=0$.

We next remind readers in the context of fractional Sobolev spaces. Given $0<s<1$, the $L^2$-based fractional Sobolev space is $H^{s}(\mathbb{R}^{n}):=W^{s,2}(\mathbb{R}^{n})$ with the norm 
\begin{equation}\notag
\|u\|_{H^{s}(\mathbb{R}^{n})}^{2}\\=\|u\|_{L^{2}(\mathbb{R}^{n})}^{2}+\|(-\Delta)^{s/2}u\|_{L^{2}(\mathbb{R}^{n})}^{2}.\label{eq:H^s norm}
\end{equation}
Furthermore, via the Parseval identity, the semi-norm $\|(-\Delta)^{s/2}u\|_{L^{2}(\mathbb{R}^{n})}^{2}$
can be rewritten as 
\[
\|(-\Delta)^{s/2}u\|_{L^{2}(\mathbb{R}^{n})}^{2}=\left((-\Delta)^{s}u,u\right)_{\mathbb{R}^{n}},
\]
where $(-\Delta)^s $ is the fractional Laplacian \eqref{fractional Laplacian}.

Let $D\subset \R^n $ be an open set and $a\in\mathbb{R}$,
then we denote the following Sobolev spaces, 
\begin{align*}
H^{a}(D) & :=\left\{u|_{D}:\, u\in H^{a}(\mathbb{R}^{n})\right\},\\
\widetilde{H}^{a}(D) & :=\text{closure of \ensuremath{C_{c}^{\infty}(D)} in \ensuremath{H^{a}(\mathbb{R}^{n})}},\\
H_{0}^{a}(D) & :=\text{closure of \ensuremath{C_{c}^{\infty}(D)} in \ensuremath{H^{a}(D)}},
\end{align*}
and 
\[
H_{\overline{D}}^{a}:=\left\{u\in H^{a}(\mathbb{R}^{n}):\,\mathrm{supp}(u)\subset\overline{D}\right\}.
\]
The fractional Sobolev space $H^{a}(D)$ is complete under the norm
\[
\|u\|_{H^{a}(D)}:=\inf\left\{ \|v\|_{H^{a}(\mathbb{R}^{n})}:\,v\in H^{a}(\mathbb{R}^{n})\mbox{ and }v|_{D}=u\right\} .
\]
Moreover, when $D$ is a Lipschitz domain, the dual spaces can be expressed as 
\begin{align*}
(H^s_{\overline{D}}(\R^n))^\ast = H^{-s}(D), \quad \text{ and }\quad (H^s(D))^\ast=H^{-s}_{\overline D}(\R^n).
\end{align*}
If reader are interested in the properties of fractional Sobolev spaces, we refer readers to the references \cite{di2012hitchhiks,mclean2000strongly}.

\subsection{The exterior Dirichlet problem}
For $m\geq 2$, $m\in \N$ and $0<s<1$. Let $\Omega \subset \R^n$ be a bounded domain with $C^{1,1}$ boundary for $n\geq 1$, and let $q(x)\in L^\infty(\Omega)$. Let us prove the well-posedness for the exterior Dirichlet problem 
	\begin{align}\label{Dirichlet problem for the well-posedness}
\begin{cases}
(-\Delta)^s u + qu^m =0 & \text{ in }\Omega ,\\
u=f &\text{ in } \Omega_e ,
\end{cases}
\end{align}
under the condition that $\norm{f}_{C^\infty_c(\Omega_e)}$ is sufficiently small. 
	
\begin{prop}[Well-posedness]\label{Prop:well posedness}
	Let $\Omega\subset \R^n$, $n\geq 1$ be a bounded domain with $C^{1,1}$ boundary $\p\Omega$, and $0<s<1$.
	Suppose that $q=q(x)\in L^\infty(\Omega)$, $m\geq 2$ and $m\in \N$.
	Then there exists $\varepsilon>0$ such that when
	\begin{align}\label{small boundary}
	f\in \mathcal{E}_\vareps:=\left\{f\in C^\infty_c(\Omega_e) :\ \norm{f}_{C^\infty_c(\Omega_e)}\leq \vareps \right\},
	\end{align} 
	the boundary value problem \eqref{Dirichlet problem for the well-posedness} has a unique solution $u$. Furthermore, the following estimate holds 
    \begin{align}\label{Bound of well-posedness}
    	\|u\|_{C^{s}(\R^n)} \leq C \|f\|_{C^\infty_c(\Omega_e)},
    \end{align}
    for some constant $C>0$, independent of $u$ and $f$.
\end{prop}

\begin{proof}
	Suppose the smallness condition that $\|f\|_{C^\infty_{c}(\Omega_e)}\leq \varepsilon$, for some small number $\vareps>0$, which will be determined later. Following the ideas of \cite[Theorem 2.1]{LL2020inverse}, one can extend $f$ to the whole space $\R^n$ by zero so that $\|f\|_{C^\infty_{c}(\R^n)}\leq \varepsilon$. Let $u_0$ be the solution of the linear Dirichlet problem 
	\begin{align}\label{eqn:u0}
	\begin{cases}
	(-\Delta)^s u_0 =0 & \hbox{ in } \Omega,\\
	u_0=f  &  \hbox{ in } \Omega_e.\\
	\end{cases} 
	\end{align}
	It is easy to see that \eqref{eqn:u0} is well-posed, i.e., there exists a unique solution $u_0 \in H^s(\R^n)$ of \eqref{eqn:u0}.
	
	Let us consider the function $w_0:=u_0-f$, then $w_0\in \wt H^s(\Omega)$ is the solution of 
	\begin{align}\label{eqn:w_0}
	(-\Delta)^s w_0 =-(-\Delta)^s f \quad \text{ in }\Omega. 
	\end{align}
	Notice that $(-\Delta)^s f$ is also bounded since $\norm{(-\Delta)^sf}_{L^\infty(\R^n)}\leq  C\norm{f}_{C^\infty_c(\R^n)}$ for some constant $C>0$ independent of $f$, by applying the optimal global H\"older regularity~\cite[Proposition~1.1]{ros2014dirichlet} to the equation \eqref{eqn:w_0} in the bounded $C^{1,1}$ domain $\Omega$, we have
	$$
	\|w_0\|_{C^{s }(\R^n)}\leq C  \norm{f}_{C^\infty_c(\R^n)}, 
	$$
	for some constant $C>0$ independent of $u_0$ and $f$. Via the triangle inequality, we have that 
	\begin{align}\label{estimate of u_0}
	\norm{u_0}_{C^s(\R^n)}\leq \|w_0\|_{C^{s }(\R^n)}+\|f\|_{C^{s }(\R^n)}\leq C\norm{f}_{C^\infty_c(\R^n)},
	\end{align}
	which shows that the solution $u_0$ of \eqref{eqn:u0} is $C^s(\overline{\Omega})$-continuous.

	If $u$ is the solution to \eqref{Dirichlet problem for the well-posedness}, then $v:=u-u_0$ satisfies
	\begin{align}\label{eqn:v}
	\begin{cases}
	(-\Delta)^s v = G(v) & \hbox{ in } \Omega,\\
	v=0  &  \hbox{ in } \Omega_e.\\
	\end{cases}
	\end{align}
	where we denote the operator $G$ by
	\begin{align}\label{definition G}
	G(\phi):=-q(u_0+ \phi)^m.
	\end{align}
	Consider the complete metric space 
	$$\mathcal{M}=\left\{\phi\in C^{s}(\R^n):\ \phi|_{\Omega_e}=0,\ \|\phi\|_{C^{s}(\R^n)}\leq \delta \right\},$$
	where $\delta>0$ will be determined later. 
	We first observe that 
	$G(\phi)=-q(u_0+\phi)^m \in  L^\infty(\Omega)$, when the functions $\phi\in\mathcal{M}$ and $q\in L^\infty(\Omega)$.
   The reason can be seen by  
	\begin{align}\label{pointwise estimate 1}
	 \norm{G(\phi)}_{L^\infty(\Omega)} \leq C\norm{q}_{L^\infty(\overline{\Omega})}\norm{u_0+\phi}^m_{L^\infty(\Omega)} <\infty,
	\end{align}
     which means $G(\phi)\in  L^\infty(\Omega)$.

	We next want to derive the following result. Let $g\in L^\infty(\Omega)$, then there exists a unique solution $\widetilde{v}\in H^s(\R^n)$ to the source problem 
	\begin{align}\label{zero boundary value problem}
	\begin{cases}
	(-\Delta)^s \widetilde{v} =g & \text{ in }\Omega, \\
	\widetilde{v}=0 & \text{ in }\Omega_e.
	\end{cases}
	\end{align}
	By \cite[Proposition~1.1]{ros2014dirichlet} again, we have 
	\begin{align}\label{Ros-Oton estimate}
	\|\widetilde{v}\|_{C^{s}(\R^n)}\leq C\|g\|_{L^\infty(\Omega)},
	\end{align}
	for some constant $C>0$ independent of $g$ and $\widetilde{v}$.

	Let us consider the solution operator of \eqref{zero boundary value problem} 
	$$
	\mathcal L_s^{-1}:  L^\infty(\Omega)\rightarrow  C^s(\R^n), \qquad g|_{\Omega}\mapsto  \widetilde{v}|_{\Omega},
	$$ 
	then we will show that the operator 
	$$\mathcal{F}:=\mathcal{L}_s^{-1}\circ G$$ 
	is a contraction map on $\mathcal{M}$.
	Assuming that $\mathcal{F}$ is contraction, by applying the contraction mapping principle on a complete metric space $\mathcal{M}$, there must exist a fixed point $v \in \mathcal{M}$  such that this fixed point $v$ is the solution of \eqref{eqn:v}.	
	To this end, we claim that $\F: \mathcal{M}\to \mathcal{M}$ and $\F$ is a contraction mapping.

	First, by \eqref{definition G} and \eqref{Ros-Oton estimate}, one has 
	\begin{align*}
	\begin{split}
	\norm{\F(\phi)}_{C^{s}(\R^n)}
	&\leq C\|G(\phi)\|_{L^\infty(\Omega)}\\
	&\leq C \norm{q | \phi+u_0|^m}_{L^\infty(\Omega)} \\
	&\leq C \|u_0 +\phi\|^m_{L^\infty(\Omega)}\\
	& \leq C\norm{u_0+\phi}^m_{C^s(\overline{\Omega})} \\
	& \leq C(\delta+\vareps)^m,
	\end{split}
	\end{align*}
	for any $\phi \in \mathcal{M}$, for some constant $C>0$ independent of $\phi$ and $u_0$.
	In addition, one can also obtain that  
	\[
	\norm{\F(\phi)}_{C^{s}(\R^n)} \leq C (\varepsilon+\delta)^m< \delta,
	\]
	where we have used $m\geq 2$ and $m\in \N$ such that $\F$ maps $\mathcal{M}$ into  $\mathcal{M}$  itself for $\vareps$ and $\delta$ small enough.

	Second, we want to show that $\F$ is a contraction mapping. By using straightforward computations, from \eqref{Ros-Oton estimate} and the mean value theorem, we have 
	\begin{align*}
	\begin{split}
	&\hskip.45cm \norm{\F(\phi_1)-\F(\phi_2)}_{C^{s}(\R^n)}  \\ 
	&=  \|(\mathcal{L}_s^{-1}\circ G)(\phi_1) -(\mathcal{L}_s^{-1}\circ G)(\phi_2)\|_{C^{s}(\R^n)}   \\
	&\leq  C \|G(\phi_1) -G(\phi_2) \|_{L^\infty(\Omega)}  \\
	&\leq  C \norm{(\phi_1 +u_0)^m-(\phi_2 +u_0)^m}_{L^\infty(\Omega)} \\
	& \leq C \norm{m |u_0+\theta\phi_1 +(1-\theta)\phi_2|^{m-1} |\phi_2-\phi_1|}_{L^\infty(\Omega)} \\
	& \leq C \norm{u_0+\theta \phi_1 +(1-\theta)\phi_2}^{m-1}_{L^\infty(\Omega)}\norm{\phi_2-\phi_1}_{C^s(\overline{\Omega})},
	\end{split} 
	\end{align*}
	for some $0<\theta<1$ and for some constant $C>0$ independent of $u_0$, $\phi_1$ and $\phi_2$.
	This infers that 
	$$
	\norm{\F(\phi_1)-\F(\phi_2)}_{C^{s}(\R^n)}\leq C(\varepsilon +\delta)^{m-1}\|\phi_2 -\phi_1\|_{C^{s}(\R^n)}
	$$
	for some constant $C_0>0$ independent of $\phi_1,\phi_2$, $\vareps$ and $\delta$.
	Now, by choosing $\varepsilon, \delta$ sufficiently small, we obtain that $C(\varepsilon +\delta)^{m-1}<1$ due to $m\geq 2$, which concludes that $\F$ is a contraction mapping on the complete metric space $\mathcal{M}$.

	In summary, there must exist a unique solution $v\in \mathcal{M}$ to the equation \eqref{eqn:v}, such that $v$ satisfies
	\begin{align*} 
	\|v\|_{C^{s}(\R^n)}\leq   & C \left(\|u_0\|^m_{C^s(\overline\Omega)} + \|v\|^m_{C^{s}(\overline\Omega)}\right) \\
	\leq  &C\LC  \vareps ^{m-1}\|f\|_{C^\infty_c(\Omega_e)} + \delta^{m-1}\|v\|_{C^{s}(\overline\Omega)}  \RC,
	\end{align*} 
	where we have used \eqref{estimate of u_0}.
	For $\delta$ sufficiently small and $m\geq 2$, $m\in \N$, we can then derive that 
	\begin{align}\label{estimate of v}
	\|v\|_{C^{s}(\R^n)}
	\leq  C \|f\|_{C^\infty_c(\Omega_e)}.
	\end{align} 
	Meanwhile, via \eqref{estimate of u_0} and \eqref{estimate of v}, the solution $u=u_0+v \in C^{s}(\R^n)$ to the Dirichlet problem \eqref{Dirichlet problem for the well-posedness} satisfies the desired estimate 
	\begin{align*} 
	\|u\|_{C^{s}(\R^n)}
	\leq  C  \|f\|_{C^\infty_c(\Omega_e)},
	\end{align*} 
	for some constant $C>0$ independent of $u$ and $f$. This completes the proof of the well-posedness.
\end{proof}

\begin{rmk}\label{rmk of H^s}
	Via Proposition \ref{Prop:well posedness}, as a matter of fact, one can that the solution $u$ of \eqref{Dirichlet problem for the well-posedness} is in $C^s_c(\R^n)$, when $u=f \in C^\infty_c(\Omega_e)$ in $\Omega_e$. Moreover, we can derive that $u\in H^s(\R^n)$, by the following straightforward computations.
	Consider the function $w=u-f$, where $u$ is the solution of \eqref{Dirichlet problem for the well-posedness} and $f\in C^\infty_c(\Omega_e)\subset C^\infty_c(\R^n)$. Then $w$ is the solution of 
	\begin{align}\label{equation w}
	\begin{cases}
	(-\Delta)^s w +w=-qu^m+u-f-(-\Delta)^s f &\text{ in }\Omega,\\
	w=0&\text{ in }\Omega_e.
	\end{cases}
	\end{align}
 By multiplying $w$ to \eqref{equation w} and integrating over $\R^n$, we can derive that $w\in H^s(\R^n)$, where we have utilzed the estimate \eqref{Bound of well-posedness}. Hence, $u = w+f \in H^s(\R^n)$. 
\end{rmk}
	We can define the DN map rigorously as follows.
	
	\begin{prop}[The DN map]\label{prop:DNmap} 
		Let $\Omega\subset \R^n$ be a bounded domain with $C^{1,1}$ boundary $\p \Omega$ for $n\geq 1$, $0<s<1$ and let $q\in L^\infty(\Omega)$.
		Define
		\begin{equation}\label{equivalent integration by parts}
		\left\langle \Lambda_q f,\varphi\right\rangle := \int_{\R^n}(-\Delta)^{s/2}u_f (-\Delta)^{s/2}\varphi\, dx + \int_{\Omega}qu_f^m\varphi \, dx,
		\end{equation}
		for $ f,\varphi\in C^\infty_c(\R^n)$.
	    The function $u_f\in C^s(\R^n)\cap H^s(\R^n)$ is the solution of \eqref{Dirichlet problem for the well-posedness} with the sufficiently small exterior data $f\in  C^\infty_c(\Omega_e)$. Then the DN map 
		\[
		\Lambda_q:H^{s}(\Omega_e)\to H^{s}(\Omega_e)^\ast 
		\]
		is bounded, and 
		\begin{equation}
		\left.\Lambda_q f\right|_{\Omega_e}=\left. (-\Delta)^s u_f\right|_{\Omega_e}.
		\end{equation}
	\end{prop}
	\begin{proof}
		Notice that \eqref{equivalent integration by parts} is not a bilinear form, since $qu^{m}$ is not a linear function. By using the Parseval identity, we have that 
		\begin{align*}
		&\int_{\R^n}(-\Delta)^{s/2}u_f (-\Delta)^{s/2}\varphi \, dx + \int_{\Omega}qu_f^{m}\varphi \, dx\\
		=& \int_{\R^n}(-\Delta)^{s}u_f \varphi \, dx + \int_{\Omega}qu_f^{m}\varphi \, dx \\
		=& \int_{\Omega_e}(-\Delta)^{s}u_f \varphi \, dx,
		\end{align*}
		where we have utilized that $u_f \in H^s(\R^n)$ is the solution of \eqref{Dirichlet problem for the well-posedness} as in Remark \ref{rmk of H^s} and 
		\[
		\int_{\R^n} (-\Delta)^{s}u_f \varphi\, dx = \int_{\Omega}(-\Delta)^{s}u_f \varphi \, dx +\int_{\Omega_e}(-\Delta)^{s}u_f \varphi\, dx.
		\]
		The preceding identity was justified in \cite{ghosh2016calder}.
		Since $\varphi\in C^\infty_c(\R^n)\subset H^s(\R^n)$ is arbitrary, by the duality argument, then we prove the proposition.
	\end{proof}

Notice that for the nonlinearities $q(x)u^{m}$, the higher order linearizations of the exterior DN map $\Lambda_q$ is particularly simple (see \cite[Section 2]{LLLS2019nonlinear} for the local case $s=1$). It is slightly different from the earlier work \cite{LLLS2019nonlinear}, which adapts multiple small parameters to do the higher order linearization. Instead, we use the ideas from \cite{LL2020inverse}, via a single $\eps$ parameter to do the higher order linearization for fractional semilinear equations.

Let $\eps>0$ be a sufficiently small number, and $g\in C^\infty_c(\Omega_e)$. 
The next proposition demonstrates that we may differentiate the fractional semilinear equation 
\begin{equation}\label{formal_calc_eq}
\begin{cases}
(-\Delta)^s u + q(x) u^m= 0 &\text{ in }\Omega,\\
 u = \eps g& \text{ in }\Omega_e,
\end{cases}
\end{equation}
formally in the $\eps$ variable to have equations corresponding to first linearization and $m$-th linearization that 
\begin{align}\label{s-harmonic}
	\begin{cases}
	(-\Delta )^s v_{g}=0 & \text{ in }\Omega,\\
	v_g =g & \text{ in }\Omega_e,
	\end{cases}
\end{align}
and  
\begin{align}\label{equ of w}
	\begin{cases}
	  (-\Delta)^s w=-(m!)q(v_{g})^{m+1} & \text{ in }\Omega,\\
	  w=0 & \text{ in }\Omega_e,
	\end{cases}
\end{align}
respectively. We call the solution $v_g$ of the fractional Laplacian equation \eqref{s-harmonic} to be $s$-harmonic in the rest of paper.

The DN map of the solution $w$ of \eqref{equ of w} is the $m$-th linearization of the DN map of~\eqref{formal_calc_eq}. Let 
\[
(D^k T)_x(y_1,\ldots,y_k) 
\]
denote the $k$-th derivative at $x$ of a mapping $T$ between Banach spaces, which can be regarded as a symmetric $k$-linear form acting on $(y_1,\ldots,y_k)$. We refer to  \cite[Section 1.1]{hormander1983analysis}, where the notation $T^{(k)}(x; y_1, \ldots, y_k)$ is used instead of $(D^k T)_x(y_1,\ldots,y_k)$.

For $f \in C^\infty_c(\Omega_e)$ with $\norm{f}_{C^\infty_c(\Omega_e)}$ to be sufficiently small.
By using the notation of $s$-harmonic function $v_g$ given by \eqref{s-harmonic}, we have the following result.

\begin{prop}\label{Prop: derivs_and_integral_formula}
    Let $\Omega\subset \R^n$ be a bounded domain with $C^{1,1}$ boundary $\p \Omega$ for $n\geq 1$, $0<s<1$ and let $q\in L^\infty(\Omega)$. Let $\Lambda_q$ be the DN map for the fractional semilinear elliptic equation 
	\begin{equation}\label{Dir_prblm}
	(-\Delta)^su + qu^m= 0 \text{ in }\Omega,
	\end{equation}
	where $m\in \N \text{ and } m\geq 2$.
	The first linearization $(D\Lambda_q)_0$ of $\Lambda_q$ at $g=0$ is the DN map of the fractional Laplacian \eqref{s-harmonic} such that 
	\[
	(D\Lambda_q)_0: H^s(\Omega_e)\to H^s(\Omega_e)^\ast, \ \ g \mapsto \left.(-\Delta)^s v_g \right|_{\Omega_e}.
	\]
	The higher order linearizations $(D^j \Lambda(q))_0$ are identically zero for $2 \leq j \leq m-1$. 
	
	The $m$-th linearization $(D^m \Lambda_q)_0$ of $\Lambda_q$ at $g=0$ can be characterized by 
	\begin{equation} \label{dm_lambdaq_identity}
	\int_{\Omega_e} (D^m \Lambda_q)_0(g,\ldots,g)h  \,dx = (m!) \int_{\Omega}  q (v_g)^{m}v_h\,dx,
	\end{equation}
     where $v_{g}$ and $v_h$ are $s$-harmonic in $\Omega$ with the exterior value $v_g=g$ and $v_h=h$ in $\Omega_e$, respectively.
\end{prop}

\begin{rmk} We point out:
	\begin{itemize}
		\item[(a)] Even though the original DN map $\Lambda_q$ depends on $q$ \emph{non-linearly}, but it is worth emphasizing that the integral identity \eqref{dm_lambdaq_identity} implies that the $m$-th order derivative of $\Lambda_q$ depends \emph{linearly} on $q$.
		
		\item[(b)] Proposition \ref{Prop: derivs_and_integral_formula} plays an essential role to prove our main results of this article.
	\end{itemize}
\end{rmk}

\begin{proof}[Proof of Proposition \ref{Prop: derivs_and_integral_formula}]
	Via Proposition \ref{Prop:well posedness}, and thus the DN map $\Lambda_q (f)= (-\Delta)^s (Sf)|_{\Omega_e}$ is well defined for sufficiently small exterior data $f$, where $S: f \mapsto u_f$ is the solution operator for the equation~\eqref{Dir_prblm}. In order to compute the derivatives of $\Lambda_q$ at $0$, it suffices to consider the derivatives of $S$. Furthermore, by using Proposition \ref{Prop:well posedness}, the maps 
	\begin{align*}
	S: \mathcal{E}_\delta \to H^s(\R^n), \qquad & f\mapsto u_f,\\
	\Lambda_q: \mathcal{E}_\delta \to H^s(\Omega_e)^\ast , \quad & f\mapsto \left. (-\Delta)^s u_f  \right|_{\Omega_e}
	\end{align*}
	are $C^\infty$ Fr\'echet differentiable mappings, where $\mathcal{E}_\delta$ is the set defined by \eqref{small boundary} to denote the set of small exterior data.
	
	Let us write $f =f(x;\eps):= \eps g(x)\in C^\infty_c(\Omega_e)$, then the function $u_{\eps g}=S(\eps g) \in C^{s}(\overline{\Omega})$ depends smoothly on the small parameter $\eps$. 	
	By applying $\left.\p_{\eps}^m \right|_{\eps=0}$ to the Taylor's formula for $C^\infty$ Fr\'echet differentiable mappings (see e.g.\ \cite[Equation 1.1.7]{hormander1983analysis})
	\[
	S(f) = \sum_{j=0}^k \frac{(D^j S)_0(f, \ldots, f)}{j!} + \int_0^1 \frac{(D^{k+1} S)_{tf}(f, \ldots, f)}{k!} (1-t)^k \,dt
	\]
	implies that $(D^k S)_0$ may be computed using the formula 
	\[
	(D^k S)_0(f, \ldots, f) = \left.  \p_{\eps}^m u_f\right|_{\eps=0}.
	\]
	Moreover, since $u_f$ is smooth in the $\eps$ variables and the fractional Laplacian $(-\Delta)^s$ is linear, one may differentiate the equation 
	\begin{equation} \label{mth_power_nonlinearity}
	(-\Delta)^s u_f + q u_f^m = 0 \text{ in }\Omega, \qquad u_f = f \text{ in }\Omega_e
	\end{equation}
	with respect to the $\eps$ variable.
	
	For the first linearization $k = 1$ with $u = u_{\eps g}$, we have $u_0 = 0$ in $\R^n$ and $m \geq 2$, the derivative of \eqref{mth_power_nonlinearity} in $\eps$ evaluated at $\eps= 0$ satisfies 
	\[
	(-\Delta)^s \left(\left.\p_{\eps}\right|_{\eps = 0} u_f\right) = 0 \text{ in }\Omega, \qquad \left.\p_{\eps}\right|_{\eps = 0} u_f= g \text{ in }\Omega_e.
	\]
	Thus the first linearization of the map $S$ at $f=0$ ($f=\eps g$ with $\eps=0$) is 
	\[
	(DS)_0(g) = \left.\p_{\eps}\right|_{\eps = 0} u_{\eps g}= v_{g}, \quad  \text{ for }g\in C^\infty_c (\Omega_e),
	\]
	where $v_{g}$ is $s$-harmonic in $\Omega$ with $v_g=g$ in $\Omega_e$.
	
	For $2 \leq k \leq m-1$, applying the $k$-th order derivatives $\left.\p_{\eps}^k \right|_{\eps =0}$ to \eqref{mth_power_nonlinearity} gives that 
	\[
	(-\Delta)^s \left(\left.\p_{\eps}^k \right|_{\eps =0}u_f\right)= 0 \text{ in }\Omega, \qquad \left.\p_{\eps}^k \right|_{\eps =0}u_f= 0 \text{ in }\Omega_e,
	\]
	since $\p_\eps ^k \left(q(x)u_f^{m}\right) $ is a sum of terms containing positive powers of the solution $u_f$, which are equal to zero whenever $\eps = 0$. The uniqueness of solutions for the fractional Laplace equation implies that 
	\[
	\underbrace{(D^k S)_0}_{k\text{-linear form}}\underbrace{(g,\ldots ,g)}_{k\text{-vector}} = 0, \quad  \text{ for }2 \leq k \leq m-1.
	\]
	More precisely, we have used the fact that any $s$-harmonic function with $0$ exterior data is zero in $\R^n$.
	
	When $k=m$, the only nonzero term in the expansion of $\left. \p_{\eps} ^m \right|_{\eps=0} \left(q(x) u_f^{m} \right)$ does not contain second or higher order derivatives of $u_f$ with respect to $\eps$. The nonzero term after inserting $\eps =0$ is 
	$$
	q(x) (m!) \left(\left.\p_{\eps}\right|_{\eps=0} u_f \right)^m=q(x)(m!)(v_g)^m.
	$$
    Hence, the function 
	\[
	%u^{(m)}  
	w:= (D^m S)_0(g, \ldots, g) = \left.\p_{\eps}^m\right|_{\eps=0}u_f \quad \text{ in }\R^n
	\]
	solves 
	\begin{equation} \label{laplace_um_equation}
	(-\Delta)^s w +q(x) (m!) (v_g)^m=0 \quad \text{ in }\Omega,
	\end{equation}
	with zero exterior data in $\Omega_e$.
	
	By linearity we have  
	\[
	\left. (D^k \Lambda_q)_0 \right|_{\Omega_e} = \left. (-\Delta)^s (D^k S)_0\right|_{\Omega_e}.
	\]
	The claims for derivatives of DN map $(D^k \Lambda_q)_0$ when $1 \leq k \leq m-1$ follow immediately. For $k = m$ we observe that  $(D^m \Lambda_q)_0(g,\ldots,g) = \left. (-\Delta)^s  w\right|_{\Omega_e}$ satisfies 
	\begin{align}\label{some integral id in sec 2}
		\begin{split}
			\int_{\Omega_e} ((-\Delta)^s w)h\,dx =& \int_{\R^n} (-\Delta)^s w v_h  \,dx+m!\int_{\Omega}q v^{m}v_h\, dx\\
		=&m!\int_{\Omega}q v^{m}v_h\, dx,
		\end{split}
	\end{align}
	where $v_h$ is $s$-harmonic in $\Omega$ with $v_{h}=h$ in $\Omega_e$. 
	Finally, we have used that $ \int_{\R^n} (-\Delta)^s w h  \,dx=0$ in \eqref{some integral id in sec 2} due to the Parseval's identity that 
	\[
	 \int_{\R^n} (-\Delta)^s w h  \,dx=\int_{\R^n}w (-\Delta)^s h \, dx =\underbrace{\int_{\Omega}w (-\Delta)^s h \, dx}_{(-\Delta)^s h =0 \text{ in }\Omega} +\underbrace{\int_{\Omega_e} w (-\Delta)^s h \, dx}_{w=0 \text{ in }\Omega_e} =0.
	\]
	Thus, the proposition follows by using \eqref{laplace_um_equation}.
\end{proof}

For the sake of convenience, in the rest of this paper, let us utilize the pairing notation 
\begin{align*}
		\langle \underbrace{ (D^m \Lambda_q)_0}_{m\text{-form}}\underbrace{(g,\ldots ,g )}_{m\text{-vectors}} ,h \rangle =\int_{\Omega_e} (D^m \Lambda_q)_0(g,\ldots,g)h  \,dx,
\end{align*}
where $(D^m \Lambda_q)_0:H^s(\Omega_e)^m \to H^s(\Omega_e)^\ast $ is regarded as an $m$-form acting on an $m$-vector valued function $(g,\ldots ,g )$.

Let $\Omega\subset \R^n$, $n\geq 1$ be a bounded domain with $C^{1,1}$ boundary $\p\Omega$, and $0<s<1$.
Let $q_1, q_2 \in L^\infty(\Omega)$, 
and $\Lambda_{q_j}$ be the DN maps of the semilinear elliptic equations $(-\Delta)^s u + q_j u^{m}=0$ in $\Omega$ for $j=1,2$. 
As we mentioned in Proposition \ref{Prop: derivs_and_integral_formula}, there are no information of $(D^k\Lambda_{q_j})_0$ for any $k=2,\cdots,m-1$. Let us look at the case $k=0$ and $k=1$. For $k=0$, we have that 
$$
(D^0 \Lambda_{q_j})_0 = \Lambda_{q_j}(\eps g)|_{\eps=0}=0, \quad \text{ for }j=1,2,
$$
due to the well-posedness of \eqref{Dirichlet problem for the well-posedness}. Meanwhile, for the case $k=1$, the map $(D^1 \Lambda_{q_j})_0$ denotes the DN map of the fractional Laplacian equation \eqref{s-harmonic}, for $j=1,2$, which has no unknown coefficients in the equation \eqref{s-harmonic}. Hence, we must have 
$$
(D^1 \Lambda_{q_1})_0=(D^1 \Lambda_{q_2})_0.
$$
Therefore, in order to understand the relations of the DN maps $\Lambda_{q_j}$, one can obtain the information of the $m$-th order derivative $(D^m\Lambda_{q_j})_0$ of the DN map $\Lambda_{q_j}$, for $j=1,2$.

\section{Monotonicity and localized potentials}\label{Sec 3}
In this section, we show monotonicity relations between potentials $q$ and their corresponding DN maps, and we demonstrate how to control the energy terms in the monotonicity formulas with the localized potentials of the fractional Laplacian.

\subsection{Monotonicity relations} 
We study the monotonicity relations between the $m$-th order derivative of DN maps and the potentials via the following integral identity.
Let us define the energy inequalities of the $m$-th order derivative of DN maps:

\begin{defi}\label{def of monotonicity}
	Consider $\Omega\subset \R^n$, $n\geq 1$ to be a bounded domain with $C^{1,1}$ boundary $\p\Omega$, and $0<s<1$.
	Let $q_1, q_2 \in L^\infty(\Omega)$, $m\geq 2$ and $m\in \N$. 
	Let $\Lambda_{q_j}$ be the DN maps of the semilinear elliptic equations $(-\Delta)^s u + q_j u^{m}=0$ in $\Omega$ for $j=1,2$. Then the inequality 
	$(D^m\Lambda_{q_1})_0\geq (D^m \Lambda_{q_2})_0$ can be defined as follows:
	\begin{itemize}
		\item[(a)] When $m$ is odd, $(D^m \Lambda_{q_1})_0\geq (D^m\Lambda_{q_2})_0$ is denoted by 
		\begin{align}\label{def mono 1}
			\langle [(D^m \Lambda_{q_1})_0  - (D^m(\Lambda_{q_2}))_0] (g,\ldots ,g), g  \rangle  \geq 0,
		\end{align}
		for any $g\in C^\infty_c(\Omega_e)$.
		\item[(b)] When $m$ is even, $(D^m\Lambda_{q_1})_0\geq (D^m \Lambda_{q_2})_0$ is denoted by 
		\begin{align}\label{def mono 2}
			\langle [(D^m\Lambda_{q_1})_0 -(D^m \Lambda_{q_2})_0] (g,\ldots ,g), h  \rangle  \geq 0,
		\end{align}
		for any $g,h\in C^\infty_c(\Omega_e)$ with $h\geq 0$.
	\end{itemize}
\end{defi}

We next demonstrate the monotonicity relations between potentials and the $m$-th order derivatives of the DN map. Due the particular structure of the power type nonlinearities, the integral identity will imply the monotonicity formulas directly, which is a more straightforward result than its linear counterpart.

\begin{thm}[Monotonicity relations]\label{Thm monotonicity}
	Consider $\Omega\subset \R^n$, $n\geq 1$ to be a bounded domain with $C^{1,1}$ boundary $\p\Omega$, and $0<s<1$.
	Let $q_1, q_2 \in L^\infty(\Omega)$, $m\geq 2$ and $m\in \N$. 
	Let $\Lambda_{q_j}$ be the DN maps of the semilinear elliptic equations $(-\Delta)^s u + q_j u^{m}=0$ in $\Omega$ for $j=1,2$.
	Then 
	\begin{itemize}
		\item[(a)] We have the integral identity 
		\begin{align}\label{integral identity}
		\langle [(D^m\Lambda_{q_1})_0 -(D^m \Lambda_{q_2})_0](g,\ldots ,g), h  \rangle = (m!)\int_{\Omega }(q_1-q_2)(v_g)^{m}v_h\, dx
		\end{align}
		where $v_g$ and $v_h$ are $s$-harmonic in $\Omega$ with $v_g=g$ and $v_h=h$ in $\Omega_e$, respectively, for $g,h\in C^\infty_c(\Omega_e)$.
		
		\item[(b)] We have the monotonicity relation 
		\begin{align*}
			q_1 \geq q_2 \text{ in }\Omega \quad \text{ implies that }\quad  (D^m\Lambda_{q_1})_0\geq (D^m \Lambda_{q_2})_0.
		\end{align*}
	\end{itemize} 
	\end{thm}

\begin{proof}
	For (a), the proof is a simple application of Proposition \ref{Prop: derivs_and_integral_formula}. Via \eqref{dm_lambdaq_identity}, one has 
	\begin{align*}
	\langle (D^m \Lambda_{q_j})_0(g,\ldots ,g), h  \rangle =	\int_{\Omega_e} (D^m \Lambda_{q_j})_0(g,\ldots,g)h  \,dx = (m!) \int_{\Omega}  q_j (v_g)^{m}v_h\,dx,
	\end{align*}
	for $j=1,2$. By subtracting the preceding identity with $j=1$ and $j=2$, we have the desired identity \eqref{integral identity}. 
	 
	For (b),  we first show the case when $m$ is odd. Let us take $h=g\in C^\infty_c(\Omega_e)$, then the uniqueness of the fractional Laplacian implies that $v_g=v_h$ in $\Omega$.	
	By plugging $q_1-q_2\geq 0$ in $\Omega$ into \eqref{integral identity}, we must have 
	\[
	\langle [(D^m\Lambda_{q_1})_0 -(D^m \Lambda_{q_2})_0](g,\ldots ,g), g  \rangle =(m!)\int_{\Omega }(q_1-q_2)(v_g)^{m+1}\, dx \geq 0,
	\]
	where we have used that $m$ is odd so that $(v_g)^{m+1}=|v_g|^{m+1}\geq 0$ in $\Omega$. This satisfies \eqref{def mono 1} so that $(D^m\Lambda_{q_1})_0\geq (D^m \Lambda_{q_2})_0$.
	
	When $m$ is even, we take $h\in C^\infty_c(\Omega_e)$ with $h\geq 0$. Note that $v_h$ is $s$-harmonic in $\Omega$ with $v_h=h\geq 0$ in $\Omega_e$, then the maximum principle for the fractional Laplacian yields that $v_h\geq 0$ in $\Omega$ (for example, see \cite{ros2015nonlocal}). By plugging $q_1-q_2\geq 0$ in $\Omega$ into \eqref{integral identity}, we must have 
	\[
	\langle [(D^m\Lambda_{q_1})_0 -(D^m \Lambda_{q_2})_0](g,\ldots ,g), h  \rangle =(m!)\int_{\Omega }(q_1-q_2)(v_g)^{m}v_h\, dx \geq 0,
	\]
	where we have used that $m$ is even so that $(v_g)^{m}=|v_g|^{m}\geq 0$ in $\Omega$ and $v_h\geq 0$ in $\Omega$. This satisfies \eqref{def mono 2} so that $(D^m\Lambda_{q_1})_0\geq (D^m \Lambda_{q_2})_0$. This completes the proof.
\end{proof}

\begin{rmk}
	From Theorem \ref{Thm monotonicity}, we have:
	\begin{itemize}
		\item[(a)] In the proof of part (b) of Theorem \ref{Thm monotonicity}, one can see that why we need to choose different $s$-harmonic function $v_h$ in \eqref{integral identity}  so that \eqref{def mono 1} and \eqref{def mono 2} have correct sign conditions. In particular, when $h\leq 0$ in $\Omega_e$, the maximum principle (see Appendix \ref{Appendix max}) yields that $(D^m\Lambda_{q_1})_0 \leq (D^m\Lambda_{q_2})_0$, provided that $q_1 \geq q_2$ in $\Omega$. However, for general $h\in C^\infty_c(\Omega_e)$, we do not know the sign condition of the $s$-harmonic function $v_h$ in $\Omega$ so that we cannot have the monotonicity relation as in Theorem \ref{Thm monotonicity} (b).
		
		\item[(b)] In particular, when $m=1$, i.e., for the (linear) fractional Schr\"odinger equation, one can adapt \eqref{def mono 2} as the monotonicity assumption. One can see that if we do the "linearization" to the fractional Schr\"odinger equation, then the "linearized" equation is also the same fractional Schr\"odinger equation. The monotonicity relations were derived in the works \cite{harrach2017nonlocal-monotonicity,harrach2020monotonicity}.
	
	    \item[(c)] In the semilinear case, the monotonicity relation between potentials and $m$-th order derivative of DN maps is equivalent to the integral identity \eqref{integral identity}, which makes the monotonicity tests be easier for the fractional semilinear elliptic equation than their linear counterparts.
	\end{itemize}
\end{rmk}

\subsection{Localized potentials for the fractional Laplacian}
We demonstrate the existence of localized potentials for $s$-harmonic functions. 
For the fractional Laplacian, the existence of localized potentials is a simple consequence of the strong uniqueness  and Runge approximation, which was demonstrated by \cite{ghosh2016calder}.
In this work, we use slightly different settings. For the sake of completeness, let us state the the strong uniqueness, Runge approximation, and localized potentials as follows.

\begin{prop}[Strong uniqueness]\label{Prop: strong uniqueness}
	For $n\geq1$, $0<s<1$, let $v\in L^p(\mathbb{R}^{n})$ for
	some $1<p < 2$ satisfy both $v$ and $(-\Delta)^{s}v$ vanish
	in the same arbitrary non-empty open set in $\mathbb R^n$, then $v\equiv0$ in $\mathbb{R}^{n}$. 
\end{prop}

The preceding proposition was shown in the proof of \cite[Theorem 1.2]{ghosh2016calder} for the case $v\in H^a(\R^n)$ for some $a\in \R$. 
In particular, Proposition \ref{Prop: strong uniqueness} was recently proved by Covi-M\"onkk\"onen-Railo \cite[Corollary 4.5]{CMR2020unique}.

We next prove the Runge approximation, and the mathematical settings are slightly different from \cite{ghosh2016calder}. In \cite{ghosh2016calder}, the authors proved any $L^2$ functions can be approximated by solutions of the fractional Schr\"odinger equation. In this work, our aim is only to demonstrate that any $L^a$-integrable functions for $a>1$, can be approximated by a sequence of $s$-harmonic functions.

\begin{thm}[Runge approximation for the fractional Laplacian]\label{Thm:runge} 
	For $n\geq 1$, $0<s<1$, let $\Omega\subseteq\mathbb{R}^{n}$
	be a bounded domain with $C^{1,1}$ boundary $\p \Omega$, and $O\Subset\Omega_{e}=\mathbb{R}^{n}\setminus\overline{\Omega }$
	be open.  Let $m\geq 2$, $m\in \N$. Given an arbitrary $a > 1$, for every $\varphi \in L^{a}(\Omega)$ there exists a sequence $g^{k}\in C_{c}^{\infty}(O)$,
	so that the corresponding solutions $v^{k}\in H^s(\R^n)$ to 
	\begin{align*}
	(-\Delta)^{s}v^{k}=0 \text{ in \ensuremath{\Omega}},\qquad u^{k}=g^{k} \text{ in }\Omega_{e},
	\end{align*}
	satisfy that $v^{k}|_{\Omega}\to \varphi$ in $L^{a}(\Omega)$ as $k\to \infty$. 
\end{thm}

\begin{proof}
	The idea of the proof is similar to the proof of \cite[Theorem 1.3]{ghosh2016calder}, but we will use the fact that if $v$ is the solution of $(-\Delta)^s v=0$ in $\Omega$ with $v=g \in C^\infty_c(\Omega_e)$, then the well-posedness yields that $v\in H^s(\R^n)$.
	Furthermore, by using the global H\"older estimate \cite[Proposition 1.1]{ros2014dirichlet}, one has $v\in C^s(\R^n)$.
	
	In order to prove the theorem, let us consider the set 
	\begin{align*}
	\mathbb{D}=\left\{v_{g}|_{\Omega}\,;\ g\in C_{c}^{\infty}(O)\right\},
	\end{align*}
	where $v_{g}\in H^{s}(\mathbb{R}^{n})$ is the unique solution of 
	\begin{align}\label{s-harmonic v_g}
		\begin{cases}
		(-\Delta)^s v_g =0 &\text{ in }\Omega,\\
		 v_g=g &\text{ in }\Omega_e,
		\end{cases}
	\end{align}
	with $g\in C^\infty_c(\Omega_e) $. Then $\mathbb{D}$ is dense in $L^{a}(\Omega)$. Via \cite[Proposition 1.1]{ros2014dirichlet}, it is easy to see that $\mathbb{D}\subset C^s(\overline{\Omega})$ which implies $\mathbb{D}\subset  L^{a}(\Omega)$, for all $a>1$.
	By the Hahn-Banach theorem, it suffices to show that for any function
	$\varphi \in L^{r}(\Omega)$ satisfying $\int_{\Omega}\varphi v_g\, dx=0$ for any $v\in\mathbb{D}$, where $\frac{1}{r}+\frac{1}{a}=1$,
	then $\varphi\equiv0$.

	Let $\varphi$ be a such function, which means $\varphi$ satisfies
	\begin{equation}\label{eq:111111}
	\int_{\Omega}\varphi v_g\, dx=0,\quad \mbox{ for any }g\in C_{c}^{\infty}(O).
	\end{equation}
	Next, let $\phi$ be the solution of 
	\begin{align*}
		\begin{cases}
		(-\Delta)^s\phi=\varphi & \text{ in }\Omega, \\
		\phi =0 & \text{ in }\Omega_e.
		\end{cases}
	\end{align*}
    By using the $L^p$ estimate for the fractional Laplacian (see Proposition \ref{Prop: Lp estimate} and Remark \ref{Rmk of Lp estimate}), we know that $\phi \in L^p(\Omega)$ for some $p\in (1,2)$ since $\varphi\in L^r(\Omega)$ for some $r > 1$.
	
	We next claim that for any $g\in C_{c}^{\infty}(O)$,
	the following relation 
	\begin{equation}\label{eq:formal adjoint relation}
     	\int_{\Omega}\varphi v_g\, dx=-\int_{\R^n} (-\Delta)^{s/2} \phi (-\Delta)^{s/2}g\, dx
	\end{equation}
	holds. In other words, $\int_{\R^n}(-\Delta)^{s/2}\phi (-\Delta)^{s/2}w \, dx = \int_{\Omega}\varphi w\, dx $
	for any $w \in \mathbb{D} \subset  L^{a}(\Omega)$. In order to prove \eqref{eq:formal adjoint relation}, let $g\in C^\infty_c(O)$, and $v_g$ be the solution of \eqref{s-harmonic v_g}. Then by \cite[Proposition 1.1]{ros2014dirichlet}, we have $v_g\in C^s(\R^n)$ with $v_g-g\in C^s_0(\Omega)$ and 
	\begin{align*}
		\int_{\Omega}\varphi v_g \, dx=&\int_{\Omega}\varphi (v_g-g)\, dx \\
		=&\int_{\R^n} (-\Delta)^{s/2}\phi \cdot (-\Delta)^{s/2}(v_g-g) \, dx\\
		=&-\int_{\R^n}(-\Delta)^{s/2}\phi \cdot  (-\Delta)^{s/2}g\, dx,
	\end{align*}
	where we have utilized that $v_g$ is $s$-harmonic in $\Omega$ and $\phi =0$ in $\Omega_e$.

	Hence, \eqref{eq:111111} and \eqref{eq:formal adjoint relation} yield that 
	imply that 
	\[
	\int_{\R^n}(-\Delta)^{s/2} \phi (-\Delta)^{s/2}g\, dx=0, \quad \mbox{ for any }g\in C_{c}^{\infty}(O).
	\]
	Moreover, we know that $g|_\Omega=0$ due to $g\in C_{c}^{\infty}(O)$, then the Parseval's identity infers that 
	\[
	\int_{\R^n} (-\Delta)^s \phi g \, dx=0, \quad \mbox{ for any }g\in C_{c}^{\infty}(O).
	\]
	In the end, we know that $\phi\in L^p (\Omega)$ with $\phi =0$ in $\Omega_e$, which satisfies $\phi \in  L^p(\mathbb{R}^{n})$ for some $p\in (1,2)$, and 
	\[
	\phi|_{O}=(-\Delta)^{s}\phi|_{O}=0.
	\]
	By applying Proposition \ref{Prop: strong uniqueness}, we obtain $\phi\equiv0$ in $\R^n$ so that $v\equiv0$ as desired.
\end{proof}

Based on the Runge approximation, one can obtain the existence of the localized potentials immediately.

\begin{cor}[Localized potentials]\label{cor:localized_potentials} 
	For $n\geq 1$, let $\Omega\subseteq\mathbb{R}^{n}$
	be a bounded domain with $C^{1,1}$ boundary $\p \Omega$, $0<s<1$, and $O\subseteq\Omega_{e}=\mathbb{R}^{n}\setminus\overline{\Omega}$
	be an arbitrary open set. For any $a> 1$ and every measurable set $M\subseteq\Omega$, there exists
	a sequence $g^{k}\in C_{c}^{\infty}(O)$, so that the corresponding
	solutions $v^{k}\in H^{s}(\mathbb{R}^{n})$ of 
	\begin{align}\label{eq:cor_loc_pot_uk}
	(-\Delta)^{s}v^{k}=0\quad\text{ in \ensuremath{\Omega}},\quad v^{k}|_{\Omega_{e}}=g^{k},\text{ for all }k\in \mathbb N
	\end{align}
	satisfy that 
	\[
	\int_{M}|v^{k}|^{a}\, dx\to\infty\quad\text{ and }\quad\int_{\Omega\setminus M}|v^{k}|^{a}\, dx\to0 \quad  \mbox{ as }k\to \infty.
	\]
\end{cor}
\begin{proof}
	The proof is based on the Runge approximation (Theorem \ref{Thm:runge}) and the normalization argument.
	By Theorem \ref{Thm:runge}, there exists a sequence $\widetilde{g}^{k}\in C^\infty_c(\Omega_e)$
	so that the corresponding solutions $\widetilde{v}^{k}|_{\Omega}$
	converge to $\left(\dfrac{1}{|M|}\right)^{\frac{1}{a}}\chi_{M}$ in $L^{a}(\Omega)$, where $|M|$ denotes the Lebesgue measure of the measurable set $M$. This implies that
	\[
	\|\widetilde{v}^{k}\|_{L^{a}(M)}^a=\int_{M}|\widetilde{v}^{k}|^{a}\, dx\to 1,
	\quad\text{ and }\quad \|\widetilde{v}^{k}\|_{L^{a}(\Omega\setminus M)}^a=\int_{\Omega\setminus M}|\widetilde{v}^{k}|^{a}\, dx\to 0,
	\]
	as $k\to \infty$.
	
	Without loss of generality, we can assume for all $k\in \mathbb{N}$ that
	$\widetilde{v}^{k}\not\equiv0$, so that $\|\widetilde{v}^{k}\|_{L^{a}(\Omega\backslash M)}>0$
	follows due to the strong uniqueness of the fractional Laplacian (Proposition~\ref{Prop: strong uniqueness}). Assume that the normalized exterior data 
	\[
	g^{k}:=\dfrac{\widetilde{g}^{k}}{\|\widetilde{v}^{k}\|_{L^{a}(\Omega\backslash M)}^{1/a}} \in C^\infty_c(\Omega_e), 
	\]
	then the sequence of corresponding solutions $v^{k}\in C^{s}(\mathbb{R}^{n})$ of \eqref{eq:cor_loc_pot_uk} has the desired property that
	\begin{align}\label{the localized potentials}
		\begin{split}
		\norm{v^{k}}_{L^{a}(M)}^a = \frac{\|\widetilde{v}^{k}\|_{L^{a}(M)}^a}{ \|\widetilde{v}^{k}\|_{L^{a}(\Omega\setminus M)} } \to \infty,
		\text{ and }   \norm{v^{k}}_{L^{a}(\Omega\setminus M)}^a=\|\widetilde{v}^{k}\|^{a-1}_{L^{a}(\Omega\setminus M)}\to 0,
		\end{split}
	\end{align}
	as $k\to \infty$, where we have used the exponent $a>1$.
\end{proof}

\begin{rmk}
	 The construction of the localized potentials for is based on the Runge approximation for the fractional Laplacian, which is a linear fractional differential equation. Notice that one might be able to study the approximation property for the fractional semilinear elliptic equation $(-\Delta)^s u + qu^m=0$ for $m\geq 2$, $m\in \N$, however, one cannot expect the existence of the localized potential for fractional semilinear equations. The reason is due to the well-posedness (Proposition \ref{Prop:well posedness}), which requires sufficiently \emph{small} exterior data, such that the solution is small as well. Therefore, the well-posedness for the fractional semilinear elliptic equation \eqref{Main equation} is an obstruction to construct the energy concentration on any (positive) measurable region inside a given domain. This implies that the $L^a$-norm of the normalized solution (see \eqref{the localized potentials}) can be arbitrarily large in some region is impossible. 
\end{rmk}

\section{Converse monotonicity, uniqueness, and inclusion detection}\label{Sec 4}
This section consists the proof of the first main result of the work. With the localized potentials \eqref{the localized potentials} and the integral identity \eqref{integral identity} at hand, we can extend Theorem \ref{Thm monotonicity} to an if-and-only-if statement.

\subsection{Converse monotonicity and the fractional Calder\'on problem}
Let us prove the if-and-only-if monotonicity relation between the potential and the $m$-th order derivative of the DN map.

\begin{proof}[Proof of Theorem \ref{Thm: If-and-only-if monotonicity}]
	Via Theorem \ref{Thm monotonicity}, $q_1 \geq q_2$ a.e. in $\Omega$ implies $(D^m_0 \Lambda_{q_1})_0 \geq (D^m\Lambda_{q_2})_0$ (in the sense of Definition \ref{def of monotonicity}). The conclusion holds if we can show that $(D^m_0 \Lambda_{q_1})_0 \geq (D^m\Lambda_{q_2})_0$ implies $q_1\geq q_2$ a.e. in $\Omega$.
	
	Suppose that $(D^m_0 \Lambda_{q_1})_0 \geq (D^m\Lambda_{q_2})_0$ holds, then the integral identity \eqref{integral identity} yields that 
	\begin{align}\label{mono in proof of converse}
			\int_{\Omega }(q_1-q_2)(v_g)^{m}v_h\, dx \geq 0,
	\end{align}
	where $v_g=v_h$ if $m$ is odd and $v_h\geq 0$ if $m$ is even (see Definition \ref{def of monotonicity} and Theorem \ref{Thm monotonicity}). In order to show that $q_1 \geq q_2$ in $\Omega$, we prove it by a standard contradiction argument. Suppose that there exists a constant $\delta>0$ and a positive measurable subset $M\subset \Omega$ such that $q_2-q_1\geq \delta>0$ in $M$. By applying the localized potentials from Corollary \ref{cor:localized_potentials} for an appropriate exponent $a>1$, which will be determined later. Hence, there must exist a sequence $\{g^k\}$ such that the corresponding $s$-harmonic functions $v^{k}$ with $v^k=g^k$ in $\Omega_e$ satisfy
	\begin{align}\label{localized potentials in the converse}
	   \int_{M}|v^k|^a \, dx \to \infty \quad \text{ and } \quad \int_{\Omega\setminus M}|v^k|^a \, dx \to 0,
	\end{align}
	as $k\to \infty$.
	
	Combine with \eqref{mono in proof of converse}, then we have:
	\begin{itemize}
		\item[(a)] 	When $m$ is odd, we take the $s$-harmonic functions $v_g=v_h$ to be the localized potentials $\{v^k\}$ into \eqref{integral identity} such that 
		\begin{align*}
		0\leq & \int_{\Omega} (q_1-q_2)|v^k|^{m+1}\, dx \\
		 \leq &-\delta \int_{M} |v^k|^{m+1}\, dx + \norm{q_1-q_2}_{L^\infty(\Omega)}\int_{\Omega\setminus M}|v^k|^{m+1}\, dx \\
		 \to &-\infty, 
		\end{align*}
		as $k\to \infty$, where we have utilized \eqref{localized potentials in the converse} as the exponent $a=m+1$, then

		\item[(b)] When $m$ is even, we need to use the other monotonicity definition \ref{def mono 2}. In this case, we choose the exterior data $h\in C^\infty_c(\Omega_e)$, $h\geq 0$ and $h\not\equiv0$. Then by the maximum principle (Proposition \ref{Prop: max principle}) in Appendix \ref{Appendix max}, we must have $v_h>0$ in $\Omega$. Meanwhile, by using the global $C^s$ estimate for the solution to the fractional Laplacian, we have $v_h\in C^s(\R^n)$ whenever $h\in C^\infty_c(\Omega_e)$. Thus, by the continuity of $v_h$, there must exists a constant $c_h>0$ such that $v_h\geq c_h>0$ in $\overline{\Omega}$.
		
		Now, let us plug the $s$-harmonic functions $v_g$ to be the localized potentials $\{v^k\}$ and $v_h> 0$ into \eqref{integral identity} such that 
		
			\begin{align*}
		0\leq & \int_{\Omega} (q_1-q_2)|v^k|^{m}v_h\, dx \\
		\leq &-\delta c_h  \int_{M} |v^k|^{m}\, dx \\
		&+ \norm{q_1-q_2}_{L^\infty(\Omega)}\norm{v_h}_{L^\infty(\Omega)}\int_{\Omega\setminus M}|v^k|^{m}\, dx \\
		\to &-\infty, 
		\end{align*}
		as $k\to \infty$.
	\end{itemize}
     The preceding arguments yield a contradiction. This implies that that $q_1 \geq q_2$ in $\Omega$ in both cases (a) and (b). Therefore, we conclude the if-and-only-if monotonicity relations \eqref{if and only if monotonicity in Sec 4} holds.
\end{proof}

\begin{cor}\label{Cor: global uniqueness}
		Let $\Omega\subset \R^n$, $n\geq 1$ be a bounded domain with $C^{1,1}$ boundary $\p\Omega$, and $0<s<1$. Let $m\geq 2$, $m\in \N$.
	Let $q_1, q_2 \in L^\infty(\Omega)$, 
	and $\Lambda_{q_j}$ be the DN maps of the semilinear elliptic equations $(-\Delta)^s u + q_j u^{m}=0$ in $\Omega$ for $j=1,2$.
	Then we have 
	\begin{align*}
	q_1 = q_2 \text{ in }\Omega \quad \text{ if and only if } \quad (D^m_0 \Lambda_{q_1})_0 = (D^m\Lambda_{q_2})_0.
	\end{align*}
\end{cor}

\begin{proof}
	The results follows immediately from Theorem \ref{Thm: If-and-only-if monotonicity}.
\end{proof}

\begin{rmk}We want to point out that:
	\begin{itemize} 
		\item[(a)] The if-and-only-if monotonicity relations has been shown by Theorem \ref{Thm: If-and-only-if monotonicity} for general potentials $q_1,q_2\in L^\infty(\Omega)$, without any sign constraints. For the (linear) fractional Sch\"odinger equation, the monotonicity relations can be proved by using the \emph{Lowner order} (see \cite{harrach2020monotonicity,HPS2019dimension}), which involves more functional analysis techniques in the arguments. We also refer readers to the further study \cite{HPSmonotonicity} for the local case. 
		
		\item[(b)] Corollary \ref{Cor: global uniqueness} is derived via the monotonicity method (Theorem \ref{Thm: If-and-only-if monotonicity}).
		In fact, in order to determine $q_1=q_2$ in $\Omega$, one can only consider the condition of the original DN maps  $\Lambda_{q_1}=\Lambda_{q_2}$ in the exterior domain. The proof is based on the higher order linearization and the Runge approximation for the fractional Laplacian, which needs to prove $(D^m_0 \Lambda_{q_1})_0 = (D^m\Lambda_{q_2})_0$ by assuming $\Lambda_{q_1}=\Lambda_{q_2}$. For more details in different approaches, we refer the reader to \cite{LL2020inverse}. 
	\end{itemize}
\end{rmk}

\subsection{A monotonicity-based reconstruction formula}
In the end of this section, let us demonstrate a proof of the constructive uniqueness for the potential $q\in L^\infty(\Omega)$ of the fractional semilinear elliptic equation \eqref{Main equation}. Inspired by  \cite{harrach2017nonlocal-monotonicity,harrach2020monotonicity}, we will show that the potential $q\in L^\infty(\Omega)$ can be reconstructed from the DN map $\Lambda_q$ by testing $\Lambda_\psi$, where $\psi$ is a \emph{simple function}.

To this end, let $M$ be a measurable set, and $M$ is called a \emph{density one set} if it is non-empty, measurable and has Lebesgue density $1$ for all $x\in M$.
The set of density one simple functions is defined by
\begin{align*}
	\Sigma &:=\textstyle \left\{ \psi=\sum_{j=1}^m a_j \chi_{M_j}:\ a_j\in \mathbb{R},\ \text{$M_j\subseteq \Omega$ is a density one set} \right\},
\end{align*}
Notice that every simple function agrees with a density one simple function almost everywhere due to the Lebesgue's density theorem. For our purposes, it is important to control the values on measure zero sets since these values 
might still affect the supremum when the supremum is taken over uncountably many functions.

We have the following constructive global uniqueness result.

\begin{thm}\label{Thm:constructive} Let $n\geq 1$, $\Omega\subset \R^n$ be a bounded domain with $C^{1,1}$ boundary $\p \Omega$, and $s\in (0,1)$. Let $q\in L^\infty(\Omega)$ and $\Lambda_{q}$ be the DN maps of the semilinear elliptic equations $(-\Delta)^s u + q u^{m}=0$ in $\Omega$, where $m\geq 2$, $m\in \N$. A potential $q=q(x)$
	can be uniquely recovered by $(D^m\Lambda_q)_0$ via the following
	formula
	\begin{align}\label{reconstruction formula for potential}
		\begin{split}
		q(x) =&\sup\left\{ \psi(x):\ \psi\in \Sigma,\ (D^m\Lambda_\psi)_0\leq (D^m\Lambda_q)_0 \right\}\\
		&+\inf\left\{ \psi(x):\ \psi\in \Sigma,\ (D^m\Lambda_\psi)_0\geq (D^m\Lambda_q)_0 \right\},
		\end{split}
	\end{align}
	for all $x\in \Omega$.
\end{thm}

\begin{rmk}
	For the local case $s=1$ and $m=2$, the reconstruction formula for the potential $q(x)$ has been studied in \cite[Corollary 3.1]{LLLS2019nonlinear}\footnote{In fact, the reconstruction formula in \cite[Corollary 3.1]{LLLS2019nonlinear} also holds for general $m\geq 2$ with $m\in \N$ in any bounded Euclidean domain $\Omega \subset \R^n$.}. The reconstruction formula was using the known the \emph{Calder\'on exponential solutions} \cite{calderon} for the Laplace equation.
\end{rmk}

To prove Theorem~\ref{Thm:constructive}, let us adapt the following lemma which was shown in \cite[Lemma 4.4]{harrach2020monotonicity}.
\begin{lem}[Simple function approximation]\label{lemma:sup_simple_functions}
	For any function $q\in L^\infty(\Omega)$, and $x\in \Omega$ a.e., we have that
	\begin{align*}
		\max\{q(x),0\}&=\sup\{ \psi(x):\ \psi\in \Sigma \text{ with } \psi\leq q \}.
	\end{align*}
\end{lem}

With the preceding lemma at hand, we can prove Theorem \ref{Thm:constructive}.

\begin{proof}[Proof of Theorem \ref {Thm:constructive}.]
Via Lemma~\ref{lemma:sup_simple_functions} and Theorem \ref{Thm: If-and-only-if monotonicity},
the potential $q\in L^\infty(\Omega)$ can be reconstructed by 
\begin{align*}
\begin{split}
	\quad q(x)
&=\max\{q(x),0\}-\max\{ -q(x),0\}\\
&=\sup\{ \psi(x):\ \psi\in \Sigma,\ \psi\leq q \}
-\sup\{ \psi(x):\ \psi\in \Sigma,\ \psi\leq -q \}\\
&=\sup\{ \psi(x):\ \psi\in \Sigma,\ \psi\leq q \}
+\inf\{ \psi(x):\ \psi\in \Sigma,\ \psi\geq q \}\\
&=\sup\left\{ \psi(x):\ \psi\in \Sigma,\ (D^m\Lambda_\psi)_0\leq (D^m\Lambda_q)_0 \right\} \\
& \quad +\inf\left\{ \psi(x):\ \psi\in \Sigma,\ (D^m\Lambda_\psi)_0\geq (D^m\Lambda_q)_0 \right\},\\
\end{split}
\end{align*}
for almost everywhere $x\in \Omega$. This shows \eqref{reconstruction formula for potential} holds for almost everywhere $x\in \Omega$. This completes the proof.
\end{proof}

\subsection{Inclusion detection by the monotonicity test} 
In this subsection, we will prove the second main result of this paper. 
The proof is also based on the if-and-only-if monotonicity relations (Theorem \ref{Thm: If-and-only-if monotonicity}), which can be regarded as an application of the converse monotonicity relation. Recall that the testing operator $T_M:H^s(\Omega_e)^m\to H^s(\Omega)^\ast $ is defined by 
\begin{align*}
\langle (T_M )(g,\ldots, g),  h \rangle  =\int_M v_g^m v_h \, dx,
\end{align*}
where $v_g$ and $v_h$ are $s$-harmonic in $\Omega$ with $v_g=g$ and $v_h=h$ in $\Omega_e$, respectively.  

\begin{proof}[Proof of Theorem \ref{Thm:support_from_closed_sets}]
 Let $\supp(q-q_0)\subset C$, then there must exist some (large) constant $\alpha>0$ such that 
 \begin{align}\label{condition with inclusion jump}
 	-\alpha \chi _C \leq q-q_0 \leq \alpha \chi _C.
 \end{align}
 By using Theorem \ref{Thm: If-and-only-if monotonicity}, we know that \eqref{condition with inclusion jump} is equivalent to 
 \begin{align}\label{condition with inclusion DN maps}
 (D^m\Lambda_{q_0-\alpha \chi _C})_0 \leq  (D^m \Lambda_q)_0 \leq (D^m\Lambda_{q_0 +\alpha \chi_c})_0.
 \end{align}
 Furthermore, via the identity for the $m$-th order derivative of the DN map \eqref{dm_lambdaq_identity}, the elements $(D^m\Lambda_{q_0\pm \alpha \chi _C})_0 $ in \eqref{condition with inclusion DN maps} can be written as 
 \begin{align}\label{linear expansion}
 	\begin{split}
 	&\langle (D^m\Lambda_{q_0\pm \alpha \chi _C})_0 (g,\ldots, g), h\rangle \\
 	 = &\int_{\Omega} (q_0\pm \alpha \chi _C)v_g^m v_h \, dx \\
 	= & \int_{\Omega} q_0v_g^m v_h \, dx \pm \alpha \int_{C}  v_g^m v_h \, dx \\
 	=&\langle (D^m\Lambda_{q_0})_0 (g,\ldots, g), h\rangle \pm  \alpha \langle T_M (g,\ldots, g), h\rangle,
 	\end{split}
 \end{align}
 where we have used the definition \eqref{testing operator}. Combining \eqref{condition with inclusion DN maps} and \eqref{linear expansion}, one obtains
 \begin{align}\label{monotonicity condition with testing operator}
 	\begin{split}
 	 (D^m\Lambda_{q_0})_0 -\alpha T_C \leq (D^m\Lambda_q)_0 \leq (D^m\Lambda_{q_0})_0 +\alpha T_C,
 	\end{split}
 \end{align}
provided that the condition \eqref{condition with inclusion jump} holds.
 
 We next prove the converse part that if there exists some $\alpha>0$ such that \eqref{monotonicity condition with testing operator} holds, then we must have  $\supp(q-q_0)\subset C$. Suppose \eqref{monotonicity condition with testing operator} holds, then Theorem \ref{Thm: If-and-only-if monotonicity} implies that 
 \begin{align*}
 	-\alpha \chi _C\leq q-q_0 \leq \alpha \chi _C.
 \end{align*}
The above inequality already shows that $q-q_0=0$ in $\Omega \setminus C$, which infers that $\supp(q-q_0)\subset C$ as desired. Hence, the assertion is proved by the monotonicity test.
\end{proof}

Note that in the statement of Theorem \ref{Thm:support_from_closed_sets}, we do not need to assume the definite case, i.e., either $q\geq q_0$ or $q\leq q_0$ in $\Omega$. We will demonstrate that it is enough to test open sets to reconstruct the \emph{inner support} for either $q\geq q_0$ or $q\leq q_0$. 

\begin{defi}
	The inner support $\mathrm{inn}\, \mathrm{supp}(\phi)$ of a measurable function $\phi:\Omega\to \R$ is the union of all open sets $U\subseteq \Omega$, for which the essential infimum $\mathrm{ess}\, \inf_{x\in U}|\phi(x)|$ is positive.
\end{defi}

\begin{thm}\label{thm:support_from_open_balls} Let $q_0 , q \in L^\infty(\Omega)$ be potentials. For the definite case, we have:
	\begin{enumerate}
		\item[(a)] Let $q\leq q_0$. For every open set $B \subseteq \Omega$ and every $\alpha>0$
		
		\begin{align}\label{eq:supp_open_balls_a1} 
			q&\leq q_0-\alpha \chi_B \implies (D^m\Lambda_q)_0 \leq (D^m\Lambda_{q_0})_0-\alpha T_{B} \implies B \subseteq \mathrm{supp}(q-q_0).
		\end{align}
	    Thus,
		\begin{align*}
		\lefteqn{\mathrm{inn\,supp}(q-q_0)}\\
		&\subseteq \bigcup \left\{B\subseteq \Omega \text{ open ball}:\ \exists \alpha>0: (D^m\Lambda_q)_0 \leq (D^m\Lambda_{q_0})_0-\alpha T_{B}\right\}\\
		&\subseteq \mathrm{supp}(q-q_0).
		\end{align*}
		\item[(b)] Let $q\geq q_0$. For every open set $B\subseteq \Omega$ and every $\alpha>0$

		\begin{align}\label{eq:supp_open_balls_b1} 
			q\geq q_0+\alpha \chi_B \implies 
			(D^m\Lambda_q) \geq (D^m\Lambda_{q_0})_0+\alpha T_{B},
		\end{align}
		and 
		\begin{align}	\label{eq:supp_open_balls_b2}
		(D^m\Lambda_q)_0 \geq (D^m \Lambda_{q_0})_0+\alpha T_{B}  \implies
		q\geq q_0+\alpha \chi_B.
		\end{align}
		Thus,
		\begin{align*}
		& \mathrm{inn\,supp}(q-q_0)\\
		&=\bigcup \left\{B\subseteq \Omega \text{ open ball}:\ \exists \alpha>0: (D^m\Lambda_q)_0 \geq (D^m \Lambda_{q_0})_0+\alpha T_{B}\right\}.
		\end{align*}
	\end{enumerate}
\end{thm}

\begin{proof}(a) If $q\leq q_0-\alpha \chi_B$, by using Theorem \ref{Thm: If-and-only-if monotonicity} and adapting the same trick as in the proof of Theorem \ref{Thm:support_from_closed_sets}, we have that 
		\begin{align*}
		(D^m\Lambda_q)_0-(D^m\Lambda_{q_0})_0\leq  -\alpha T_{B}.
		\end{align*}
		Moreover, if $	(D^m\Lambda_q)_0 \leq (D^m\Lambda_{q_0})_0\leq  -\alpha T_{B}$, by Theorem \ref{Thm: If-and-only-if monotonicity} and Theorem \ref{Thm:support_from_closed_sets} again, that there exists $c>0$ with
		\begin{align*}
		\alpha T_{B}\leq  (D^m \Lambda_{q_0})_0-(D^m\Lambda_q)_0=\int_{\Omega}  (q_0-q)v_g^m v_h\, dx,
		\end{align*}
		which implies 
		\begin{align}\label{integral inequ in inn supp 1}
			\int_{\Omega}\left( \alpha \chi_B v_g^m v_h -\norm{q-q_0}_{L^\infty(\Omega)} \chi_{\supp(q-q_0)}  v_g^m v_h\right) dx\leq 0.
		\end{align}
	    With the localized potential for the fractional Laplacian at hand (similar to the proof of Theorem \ref{Thm: If-and-only-if monotonicity}), the inequality \eqref{integral inequ in inn supp 1} must yield 
	     \[
	     \alpha \chi _B \leq \norm{q-q_0}_{L^\infty(\Omega)} \chi _{\supp(q-q_0)}.
	     \]
		
		(b) The results \eqref{eq:supp_open_balls_b1} and \eqref{eq:supp_open_balls_b2} are simple applications of Theorem \ref{Thm: If-and-only-if monotonicity}. 		
\end{proof}

\section{Lipschitz stability with finitely many measurements}\label{Sec 5}
In the last section of this paper, we prove Theorem \ref{Thm:stability}, and the ideas of the proof are from  \cite{harrach2019uniqueness,harrach2020monotonicity}.

\begin{proof}[Proof of Theorem \ref{Thm:baby stability}]
Let us divide the proof into several steps.

\vspace{2mm}

\noindent{\it Step 1. Fundamental estimates}

\vspace{2mm}

For $q_1 \not \equiv q_2$ in $\Omega$ with $q_1,q_2\in \Q$, we want to show that 
\begin{align}\label{claim 1 in the stability}
	\frac{\norm{(D^m\Lambda_{q_1})_0 -(D^m \Lambda_{q_2})_0}_{\ast}}{\norm{q_1-q_2}_{L^\infty(\Omega)}} \geq \inf_{
		\kappa \in \mathcal{K}
		}\sup_{\begin{subarray}{c}
		g,h\in C^\infty_c(\Omega_e), \\
		\norm{g}_{H^s}=\norm{h}_{H^s}=1	\end{subarray}}  \Phi(\kappa, g, h),
\end{align}
where $\Phi: \mathcal{K} \times C^\infty_c(\Omega_e)\times C^\infty(\Omega_e)$ is given by 
\[
\Phi(\kappa, g, h):= \max \left\{ \langle (D^m\Lambda_{\kappa})_0(g,\ldots, g), h \rangle \right\},
\]
and $\mathcal{K}=\left\{\kappa \in \mathrm{span} \Q: \ \norm{\kappa}_{L^\infty(\Omega)}=1 \right\}$ is a finite-dimensional subspace of $L^\infty(\Omega)$.
By the definition of $\norm{\cdot}_\ast$, we have 
\begin{align*}
	\norm{(D^m\Lambda_{q_1})_0 -(D^m \Lambda_{q_2})_0}_{\ast} = \sup_{\begin{subarray}{c}
		g,h\in C^\infty_c(\Omega_e), \\
		\norm{g}_{H^s}=\norm{h}_{H^s}=1
		\end{subarray}} \langle  (D^m\Lambda_{q_1})_0 -(D^m \Lambda_{q_2})_0(g,\ldots, g), h\rangle
\end{align*}
and 
\begin{align*}
	&\left| \langle  (D^m\Lambda_{q_1})_0 -(D^m \Lambda_{q_2})_0(g,\ldots, g), h\rangle \right| \\
	= & \max \left\{\langle  (D^m\Lambda_{q_1})_0 -(D^m \Lambda_{q_2})_0(g,\ldots, g), h\rangle , \right. \left.\langle  (D^m\Lambda_{q_2})_0 -(D^m \Lambda_{q_1})_0(g,\ldots, g), h\rangle  \right\} \\
	=&\norm{q_1-q_2}_{L^\infty(\Omega)} \max \left\{\int_{\Omega} \frac{q_1-q_2}{\norm{q_1-q_2}_{L^\infty(\Omega)} }v_g^m v_h \, dx,  , \int_{\Omega} \frac{q_2-q_1}{\norm{q_1-q_2}_{L^\infty(\Omega)} }v_g^m v_h \, dx \right\} \\
	=&\norm{q_1-q_2}_{L^\infty(\Omega)} \Phi\left(\frac{q_1-q_2}{\norm{q_1-q_2}_{L^\infty(\Omega)} }, g, h\right),
\end{align*}
where we have utilized the integral identity \eqref{integral identity} and the linearity of the $m$-th order derivative of the DN map $\Lambda_q$ in the above computations. Therefore, the claim \eqref{claim 1 in the stability} must hold.

\vspace{2mm}

\noindent{\it Step 2. Positive lower bound of $\Phi$}

\vspace{2mm}

We next show that there exists $\widehat{\kappa}\in \mathcal{K}$ such that 
\begin{align*}
	 \inf_{\kappa \in \mathcal{K}}\sup_{\begin{subarray}{c}
		g,h\in C^\infty_c(\Omega_e), \\
		\norm{g}_{H^s}=\norm{h}_{H^s}=1	\end{subarray}}  \Phi(\kappa,  g, h)= \sup_{\begin{subarray}{c}
		g,h\in C^\infty_c(\Omega_e), \\
		\norm{g}_{H^s}=\norm{h}_{H^s}=1	\end{subarray}}  \Phi(\widehat{\kappa}, g, h).
\end{align*}
The fact directly follows by the smoothness of the DN map (see Section \ref{Sec 2}) such that the function 
\[
\kappa \mapsto  \sup_{\begin{subarray}{c}
	g,h\in C^\infty_c(\Omega_e), \\
	\norm{g}_{H^s}=\norm{h}_{H^s}=1	\end{subarray}}  \Phi(\kappa, g, h)
\]
is lower semicontinuous and its minimum can be achieved over the compact set $\mathcal{K}$ (a finite dimensional subspace of $L^\infty(\Omega)$).

Finally, since $q_1 -q_2 \not \equiv 0$ in $\Omega$, by applying the localized potentials for $s$-harmonic functions (Corollary \ref{cor:localized_potentials}), there must exist $g,h\in H^s(\Omega_e)$ such that 
\begin{align}
	\text{either }\quad \int_{\Omega}\kappa v_g^m v_h \, dx >0 \quad \text{ or }\quad \int_{\Omega}\kappa v_g^m v_h \, dx <0,
\end{align}
where we have utilized the fact that $\kappa \in \mathrm{span} \Q$.
Hence, we can obtain 
\[
 c_0:=\sup_{\begin{subarray}{c}
	g,h\in C^\infty_c(\Omega_e), \\
	\norm{g}_{H^s}=\norm{h}_{H^s}=1	\end{subarray}}  \Phi(\kappa, g, h)>0, \quad \text{ for any }\kappa \in\mathcal{K},
\]
which completes the proof.
\end{proof}

It remains to prove our last theorem in the paper.

\begin{proof}[Proof of Theorem \ref{Thm:stability}]
	By using 
	\begin{align*}
	  &\left\| P_{H_\ell}'  \left( (D^m\Lambda_{q_2})_0-(D^m\Lambda_{q_1})_0 \right)P_{H_\ell}\right\|_{\ast} \\
	  =&\sup_{g,h\in H_\ell} \left| \langle \left( (D^m\Lambda_{q_2})_0-(D^m\Lambda_{q_1})_0 \right)(g,\ldots, g), h  \rangle \right|,
	\end{align*}
	and applying the preceding arguments, for any $\ell \in \N$, there exists $\kappa_\ell\in \mathcal{K}$ such that 
	\begin{align}\label{proof 1 in the stability}
		\frac{\left\| P_{H_\ell}'  \left( (D^m\Lambda_{q_2})_0-(D^m\Lambda_{q_1})_0 \right)P_{H_\ell}\right\|_{\ast}}{\norm{q_1-q_2}_{L^\infty(\Omega)}} \geq \sup_{\begin{subarray}{c}
			g,h\in H_\ell, \\
			\norm{g}_{H^s}=\norm{h}_{H^s}=1 
			\end{subarray}}\Phi(\kappa_\ell,g,h).
	\end{align}
	Notice that the right hand side of \eqref{proof 1 in the stability} is monotonically increasing in $\ell \in \N$, since $H_1\subseteq H_2 \subseteq \ldots \subseteq H_\ell \subset \ldots \subseteq H^s(\Omega_e)$. Therefore, Theorem \ref{Thm:stability} holds if we can prove that there is $\ell \in \N$ such that 
	\begin{align}\label{proof 2 in the stability}
	\sup_{\begin{subarray}{c}
		g,h\in H_\ell, \\
		\norm{g}_{H^s}=\norm{h}_{H^s}=1 
		\end{subarray}}\Phi(\kappa, g,h) >0 , \quad \text{ for all }\kappa \in \mathcal{K}.
	\end{align}
	We prove the claim \eqref{proof 2 in the stability} by a contradiction argument, i.e., there must exist a sequence $(\kappa_\ell)_{\ell\in \N}\subset \mathcal{K}$ such that  
	\begin{align*}
	\sup_{\begin{subarray}{c}
		g,h\in H_\ell, \\
		\norm{g}_{H^s}=\norm{h}_{H^s}=1 
		\end{subarray}}\Phi(\kappa_\ell,g,h) \leq 0 , \quad \text{ for } \ell \geq m,
	\end{align*}
	for any $m\in \N$.
	After passing a subsequence (if necessary), by the compactness (due to the finite dimensional assumption of $\Q$), we can assue that there exists an element $\kappa_\infty\in \mathcal{K}$ such that $\displaystyle\lim_{\ell \to \infty} \kappa_\ell =\kappa _\infty$ and 
	\[
	\sup_{\begin{subarray}{c}
		g,h\in H_m, \\
		\norm{g}_{H^s}=\norm{h}_{H^s}=1 
		\end{subarray}}\Phi(\kappa_\infty, g,h) \leq \lim_{\ell \to \infty}\sup_{\begin{subarray}{c}
		g,h\in H_m, \\
		\norm{g}_{H^s}=\norm{h}_{H^s}=1 
		\end{subarray}}\Phi(\kappa_\ell, g,h) \leq 0,  
	\]
	where we have utilized the lower semicontinuous of the function 
	$$\kappa \to \sup_{\begin{subarray}{c}
		g,h\in H_\ell, \\
		\norm{g}_{H^s}=\norm{h}_{H^s}=1 
		\end{subarray}}\Phi(\kappa, g,h) .
	$$
	On the other hand, by the continuity, we must have 
	\[
	\Phi(\kappa_\infty,g,h)\leq 0, \quad \text{ for all }g,h\in \overline{\bigcup_{m\in \N}H_m}=H^s(\Omega_e),
	\]
	which contradicts to \eqref{claim 1 in the stability} in the previous proof. This proves that \eqref{proof 2 in the stability} must hold for some $\ell \in \N$ as desired.
\end{proof}

\appendix 

\section{The $L^p$-estimate for the fractional Laplacian}\label{Appendix} 
Let us review the other estimates for solutions to the fractional Laplacian: The $L^p$ estimate. 
Before doing so, let us recall some fundamental properties for the \emph{Riesz potential}.

\begin{prop}[Riesz potential]\label{Prop: Riesz potential}
	For $0<s<1$ with $n>2s$. Let $V$ and $F$ satisfy 
	\begin{align*}
		V=(-\Delta)^{-s}F \text{ in }\R^n,
	\end{align*}
	in the sense that $V$ is the Riesz potential of order $2s$ of the function $F$. 
	\begin{itemize}
		\item[(a)] If $F\in L^1(\R^n)$, then there exists a constant $C>0$ depending only on $n$ and $s$ such that 
		\begin{align*}
			\norm{V}_{L^p_{\mathrm{w}}(\R^n)} \leq C\norm{F}_{L^1(\R^n)}, 
		\end{align*}
		where $L^p_{\mathrm{w}}$ denotes the weak-$L^p$ norm and $p=\frac{n}{n-2s}$.
		
		\item[(b)] For $r\in (1,\frac{n}{2s})$, $F\in L^r(\R^n)$, then there exists a constant $C>0$ depending only on $n$, $s$, and $r$ such that 
		\begin{align*}
		\norm{V}_{L^p(\R^n)} \leq C\norm{F}_{L^r(\R^n)}, 
		\end{align*}
		where $p=\frac{nr}{n-2rs}$.
		
		\item[(c)] For $r\in (\frac{n}{2s}, \infty)$, then there exists a constant $C>0$ depending only on $n$, $s$, and $r$ such that 
		\begin{align*}
			[u]_{C^\alpha(\R^n)}\leq C\norm{F}_{L^r(\R^n)},
		\end{align*}
		where $\alpha=2s-\frac{n}{p}$ and $[u]_{C^\alpha (\R^n)}$ is the seminorm given in Section \ref{Sec 2}.
	\end{itemize}
\end{prop}

\begin{proof}
	Parts (a) and (b) are classical results for the Riesz potential, and the proof can be found in Stein's book \cite[Chapter V]{stein2016singular}. For (c), we refer readers to \cite[p.164]{stein2016singular} and \cite{garcia2004boundedness}.
\end{proof}

Furthermore, we have the following H\"older estimate for the fractional Laplacian, which was shown in \cite[Proposition 1.7]{RS2014extremal}. We state the result in the following proposition and the proof can be found in \cite{RS2014extremal}.

\begin{prop}[$C^\beta$-estimate]\label{Prop:exterior Holder in appendix}
	For $n\geq 1$, $0<s<1$, let $\Omega\subseteq\mathbb{R}^{n}$
	be a bounded domain with $C^{1,1}$ boundary $\p \Omega$. Let $h\in C^\alpha(\Omega_e)$ for some $\alpha \in (0,1)$. Let $w$ be the solution of 
	\begin{align*}
		\begin{cases}
		(-\Delta)^s w=0 & \text{ in }\Omega, \\
		w=h & \text{ in }\Omega_e.
		\end{cases}
	\end{align*}
	Then the solution $w\in C^\beta (\R^n)$, where $\beta =\min \{s,\alpha \}$, and 
	\begin{align*}
		\norm{w}_{C^\beta(\R^n)}\leq C\norm{h}_{C^\alpha(\Omega_e)},
	\end{align*}
	for some constant $C>0$ depending only on $\Omega$, $\alpha$, and $s$.
\end{prop}

The following proposition was also proved in \cite{RS2014extremal}, which is an important result in the proof of our Runge approximation (Theorem \ref{Thm:runge}). We state the result and prove it for the sake of completeness. The proof is based on the preceding properties of the Riesz potential, the maximum principle for the fractional Laplacian and the $C^\beta$-estimate (Proposition \ref{Prop:exterior Holder in appendix}).

\begin{prop}\label{Prop: Lp estimate}
	For $n\geq 1$, $0<s<1$, let $\Omega\subseteq\mathbb{R}^{n}$
	be a bounded domain with $C^{1,1}$ boundary $\p \Omega$. For $F\in L^r (\Omega)$, let $v$ be the solution of 
	\begin{align}\label{exterior value problem in appendix}
	\begin{cases}
	(-\Delta)^s v =F & \text{ in }\Omega, \\
	v=0 & \text{ in }\Omega_e,
	\end{cases}
	\end{align}
	then we have:
	\begin{itemize}
		\item[(a)] Let $n>2s$, $r=1$, and $p\in [1, \frac{n}{n-2s})$ be an arbitrary number,  then there exists a constant $C>0$ independent of $v$ and $F$  such that 
		\begin{align*}
		\norm{v}_{L^p(\Omega)}\leq C\norm{F}_{L^1(\Omega)}.
		\end{align*}
		
		\item[(b)] Let $n>2s$, $r\in (1,\frac{n}{2s})$ and $p=\frac{nr}{n-2rs}$, then there exists a constant $C>0$ independent of $v$ and $F$ such that 
		\begin{align*}
		\norm{v}_{L^p(\Omega)}\leq C\norm{F}_{L^p (\Omega)}.
		\end{align*}
		
		\item[(c)] Let $n>2s$, $r\in (\frac{n}{2s},\infty)$, and $\beta=\min \left\{s, 2s-\frac{n}{r} \right\}$, then there exists a constant $C>0$ independent of $v$ and $F$ such that 
		\begin{align*}
		\norm{v}_{C^\beta(\Omega)}\leq C\norm{F}_{L^r (\Omega)}.
		\end{align*}
		
		\item[(d)] Let $n=1$, $s\in [\frac{1}{2},1)$, $r\geq 1$, and any $p<\infty$, then there exists a constant $C>0$ independent of $v$ and $F$  such that 
		\begin{align*}
		\norm{v}_{L^p(\Omega)}\leq C\norm{F}_{L^r (\Omega)}.
		\end{align*}
	\end{itemize}
\end{prop}

\begin{proof}
	(a) Let us extend the function $F$ by $0$ outside $\Omega$, and we still denote the function as $F$. Let $V$ be the solution of 
	\[
	(-\Delta)^s V=|F| \text{ in }\R^n,
	\]
	so that $V=(-\Delta)^{-s}|F|$ in $\R^n$, where $(-\Delta)^{-s}|F|$ is the Riesz potential of $|F|$.
	By the definition of the Riesz potential, we have $V\geq 0$ in $\Omega_e$. Via the maximum principle, we obtain that $|v|\leq V$ in $\Omega$. By applying Proposition \ref{Prop: Riesz potential}, one can see that 
	\begin{align*}
		\norm{v}_{L^q_{\mathrm{w}}(\Omega)}\leq \norm{V}_{L^q_{\mathrm{w}}(\Omega)}\leq C\norm{F}_{L^1(\Omega)},
	\end{align*}
	if $F\in L^1(\Omega)$ and for some constant $C>0$ independent of $v$ and $F$. Thus, one has 
	\begin{align*}
		\norm{v}_{L^r(\Omega)} \leq C\norm{F}_{L^1(\Omega)},
	\end{align*}
	for some constant $C>0$ independent of $v$ and $F$. This proves (a).

	(b) Similarly, the proof of (b) can be completed by using the result (b) in Proposition \ref{Prop: Riesz potential} and the maximum principle for the fractional Laplacian as before.
	Furthermore, when $r=\frac{n}{2s}$, it is easy to see that $F\in L^{r}(\Omega)\subset L^{\wt r}(\Omega)$, for any $\wt r\in [1,r]$ (since $\Omega$ is bounded). We still have the $L^p$ estimate for the solution in the borderline case $r=\frac{n}{2s}$.

   (c) Let us write $v=\wt v+ w$, where $\wt v$ and $w$ are given by 
   \begin{align}\label{proof Cbeta in appendix 1}
   	\wt v =(-\Delta)^{-s}F \text{ in }\R^n, 
   \end{align}
   and 
   \begin{align}\label{proof Cbeta in appendix 2}
   	\begin{cases}
      (-\Delta)^s w =0 & \text{ in }\Omega, \\
      w=\wt v & \text{ in }\Omega_e.
   	\end{cases}
   \end{align}
   By using \eqref{proof Cbeta in appendix 1} and Proposition \ref{Prop: Riesz potential} (c), there exists a constant $C>0$ depending only on $n$, $s$, and $r$ such that 
   \begin{align*}
   	[\wt v]_{C^\alpha(\R^n)}\leq C\norm{F}_{L^r(\R^n)}, \quad \text{ where }\alpha=2s-\frac{n}{r}.
   \end{align*}
   Since $\Omega$ is bounded and $F$ is compactly supported, one has $\wt v$ decays at infinity. This implies that 
   \begin{align}\label{proof Cbeta in appendix 3}
   	\norm{\wt v}_{C^\alpha(\R^n)}\leq C\norm{F}_{L^r(\R^n)}, \quad \text{ where }\alpha=2s-\frac{n}{r},
   \end{align}
   for some constnat $C>0$ depending only on $n$, $s$, $r$ and $\Omega$.
   
    On the other hand, we can apply Proposition \ref{Prop:exterior Holder in appendix} to derive the H\"older estimate for the solution $w$ of \eqref{proof Cbeta in appendix 2} that 
    \begin{align}\label{proof Cbeta in appendix 4}
    	\norm{w}_{C^\beta (\R^n)}\leq C\norm{\wt v}_{C^\alpha(\Omega_e)},
    \end{align}
    for some constant $C>0$ depending only on $\Omega$, $\alpha$, and $s$, where 
    $$
    \beta=\min\{\alpha , s\}=\min\left\{s,  2s-\frac{n}{r}\right\}.
    $$ 
    Combining with \eqref{proof Cbeta in appendix 3} and \eqref{proof Cbeta in appendix 4}, we can obtain the H\"older estimate for the solution $v=\wt v+w$ such that (c) holds. Moreover, since $v\in C^\beta(\overline{\Omega})$ with $v=0$ in $\Omega_e$, we must have $v\in L^p(\R^n)$ for any $p\geq 1$.

	(d) Notice that for $s<\frac{1}{2}$, we have $n=1>2s$ automatically, such that the case (d) holds by applying the results either (a) or (b). On the other hand, for the case $1=n\leq 2s$, this implies that $\frac{1}{2}\leq s <1$. Under this situation, any bounded domain is of the form $\Omega=(a,b)\subset \R$. By \cite{blumenthal1961distribution}, the Green function $G(x,y)$ for the exterior value problem \eqref{exterior value problem in appendix} is explicit. Furthermore, $G(\cdot,y)\in L^\infty(\Omega)$ when $s>\frac{1}{2}$ and $G(x,y)\in L^r(\Omega)$ for any $r<\infty$ when $s=\frac{1}{2}$. Therefore, one has 
	\[
	\norm{v}_{L^\infty(\Omega)}\leq C\norm{F}_{L^1(\Omega)},
	\]
	for some constant $C>0$ independent of $v$ and $F$, where $n<2s$. For the case $n=2s$, we have either
	\[
	 \norm{v}_{L^p(\Omega)}\leq C\norm{F}_{L^1(\Omega)}, \quad \text{ for all }p<\infty, 
	 \]
	 or
	 \[
	  \norm{v}_{L^\infty(\Omega)}\leq C\norm{F}_{L^r(\Omega)}, \quad \text{ for }r>1,
	\]
	for some constant $C>0$ independent of $v$ and $F$. 
\end{proof}

\begin{rmk}\label{Rmk of Lp estimate}
	From the $L^p$ estimate of $s$-harmonic functions, we have:
	\begin{itemize}
		\item[(a)] 	No matter what exponent $r\geq 1$ and what space dimension $n$ are, for any $F\in L^r(\Omega)$ with $\Omega\subset \R^n$ in the statement of Proposition \ref{Prop: Lp estimate}, then we can always conclude that the solution $v$ of \eqref{exterior value problem in appendix} must belong to $L^p(\R^n)$, for some $p> 1$. 
		
		\item[(b)] Moreover, since the domain $\Omega$ is bounded in $\R^n$, then we can confine the exponent $p$ in the region $p \in (1,2)$. The condition $p \in (1,2)$ plays an essential role in order to prove Proposition \ref{Prop: strong uniqueness} (see \cite[Section 4]{CMR2020unique} for more detailed discussions about the strong uniqueness of the $s$-harmonic function). Meanwhile, we also need to use the $L^p$-estimate to prove the Runge approximation via the strong uniqueness for the fractional Laplacian in Section \ref{Sec 3}.
	\end{itemize}
\end{rmk}

\section{The maximum principle}\label{Appendix max}

We review the known maximum principle for the fractional Laplacian in the end of this work. These results were shown in \cite{ros2015nonlocal,bucur2016nonlocal} for the fractional Laplacian equation and \cite{lai2019global,LL2020inverse} for the fractional Schr\"odinger equation. For the sake of convenience, we state the results as follows.

\begin{prop}[The maximum principle]\label{Prop: max principle}
	Let $\Omega\subset \R^n$, $n\geq 1$ be a bounded domain with Lipschitz boundary $\p\Omega$, and $0<s<1$.  Let $v\in H^s(\R^n)$ be the unique solution of 
	\begin{align*}
	\begin{cases}
	(-\Delta)^s v=F & \text{ in }\Omega, \\
	v=g &\text{ in }\Omega_e.
	\end{cases}
	\end{align*} 
	Suppose that $0\leq F\in L^\infty(\Omega)$ in $\Omega$ and $0\leq g \in L^\infty(\Omega_e)$ in $\Omega_e$. Then $v\geq 0$ in $\Omega$. Moreover, if $g\not \equiv 0$ in $\Omega_e$, then $v>0$ in $\Omega$. 
\end{prop}

\vskip0.5cm

\noindent\textbf{Acknowledgments.} 
The author is grateful to Dr. Jesse Railo for many fruitful discussions.
The author is partially  supported by the Ministry of Science and Technology Taiwan, under the Columbus Program: MOST-109-2636-M-009-006, 2020-2025.

\bibliographystyle{alpha}
\bibliography{ref}

\newcommand{\etalchar}[1]{$^{#1}$}
\begin{thebibliography}{VMC{\etalchar{+}}17}

\bibitem[AH13]{arnold2013unique}
Lilian Arnold and Bastian Harrach.
\newblock Unique shape detection in transient eddy current problems.
\newblock {\em Inverse Problems}, 29(9):095004, 2013.

\bibitem[BGR61]{blumenthal1961distribution}
Robert~M Blumenthal, Ronald~K Getoor, and DB~Ray.
\newblock On the distribution of first hits for the symmetric stable processes.
\newblock {\em Transactions of the American Mathematical Society},
  99(3):540--554, 1961.

\bibitem[BHHM17]{barth2017detecting}
Andrea Barth, Bastian Harrach, Nuutti Hyv{\"o}nen, and Lauri Mustonen.
\newblock Detecting stochastic inclusions in electrical impedance tomography.
\newblock {\em Inverse Problems}, 33(11):115012, 2017.

\bibitem[BHKS18]{brander2018monotonicity}
Tommi Brander, Bastian Harrach, Manas Kar, and Mikko Salo.
\newblock Monotonicity and enclosure methods for the {$p$}-{L}aplace equation.
\newblock {\em SIAM J. Appl. Math.}, 78(2):742--758, 2018.

\bibitem[BV16]{bucur2016nonlocal}
Claudia Bucur and Enrico Valdinoci.
\newblock {\em Nonlocal diffusion and applications}, volume~20.
\newblock Springer, 2016.

\bibitem[Cal80]{calderon}
Alberto~P Calder{\'o}n.
\newblock On an inverse boundary value problem.
\newblock {\em Seminar in Numerical Analysis and its Applications to Continuum
  Physics (R\'{i}o de Janeiro: Soc. Brasileira de Matem\'{a}tica)}, pages
  65--73, 1980.

\bibitem[CLL19]{CLL2017simultaneously}
Xinlin Cao, Yi-Hsuan Lin, and Hongyu Liu.
\newblock Simultaneously recovering potentials and embedded obstacles for
  anisotropic fractional {S}chr\"odinger operators.
\newblock {\em Inverse Problems and Imaging}, 13(1):197--210, 2019.

\bibitem[CLR20]{cekic2020calderon}
Mihajlo Cekic, Yi-Hsuan Lin, and Angkana R{\"u}land.
\newblock The {C}alder{\'o}n problem for the fractional {S}chr{\"o}dinger
  equation with drift.
\newblock {\em Cal. Var. Partial Differential Equations}, 59(91), 2020.

\bibitem[CMR20]{CMR2020unique}
Giovanni Covi, Keijo M{\"o}nkk{\"o}nen, and Jesse Railo.
\newblock Unique continuation property and {P}oincar\'e inequality for higher
  order fractional {L}aplacians with applications in inverse problems.
\newblock {\em arXiv preprint arXiv:2001.06210}, 2020.

\bibitem[DFS20]{daimon2020monotonicity}
Tomohiro Daimon, Takashi Furuya, and Ryuji Saiin.
\newblock The monotonicity method for the inverse crack scattering problem.
\newblock {\em Inverse Problems in Science and Engineering}, pages 1--12, 2020.

\bibitem[DNPV12]{di2012hitchhiks}
Eleonora Di~Nezza, Giampiero Palatucci, and Enrico Valdinoci.
\newblock Hitchhiker's guide to the fractional {S}obolev spaces.
\newblock {\em Bulletin des Sciences Math{\'e}matiques}, 136(5):521--573, 2012.

\bibitem[FKSU09]{ferreira2009linearized}
David D.~S. Ferreira, Carlos Kenig, Johannes Sj\"ostrand, and Gunther Uhlmann.
\newblock On the linearized local {C}alder\'on problem.
\newblock {\em Math. Res. Lett.}, 16:955--970, 2009.

\bibitem[FO19]{FO19}
Ali Feizmohammadi and Lauri Oksanen.
\newblock An inverse problem for a semi-linear elliptic equation in
  {R}iemannian geometries.
\newblock {\em arXiv:1904.00608}, 2019.

\bibitem[Gar17]{garde2017comparison}
Henrik Garde.
\newblock Comparison of linear and non-linear monotonicity-based shape
  reconstruction using exact matrix characterizations.
\newblock {\em Inverse Problems in Science and Engineering}, pages 1--18, 2017.

\bibitem[Gar19]{garde2019reconstruction}
Henrik Garde.
\newblock Reconstruction of piecewise constant layered conductivities in
  electrical impedance tomography.
\newblock {\em arXiv preprint arXiv:1904.07775}, 2019.

\bibitem[GCG04]{garcia2004boundedness}
Jos{\'e} Garc{\'\i}a-Cuerva and A~Eduardo Gatto.
\newblock Boundedness properties of fractional integral operators associated to
  non-doubling measures.
\newblock {\em Studia Mathematica}, 162:245--261, 2004.

\bibitem[Geb08]{gebauer2008localized}
Bastian Gebauer.
\newblock Localized potentials in electrical impedance tomography.
\newblock {\em Inverse Probl. Imaging}, 2(2):251--269, 2008.

\bibitem[GH18]{griesmaier2018monotonicity}
Roland Griesmaier and Bastian Harrach.
\newblock Monotonicity in inverse medium scattering on unbounded domains.
\newblock {\em SIAM J. Appl. Math}, 78(5):2533--2557, 2018.

\bibitem[GLX17]{ghosh2017calder}
Tuhin Ghosh, Yi-Hsuan Lin, and Jingni Xiao.
\newblock The {C}alder\'{o}n problem for variable coefficients nonlocal
  elliptic operators.
\newblock {\em Communications in Partial Differential Equations},
  42(12):1923--1961, 2017.

\bibitem[GRSU20]{GRSU18}
Tuhin Ghosh, Angkana R{\"u}land, Mikko Salo, and Gunther Uhlmann.
\newblock Uniqueness and reconstruction for the fractional {C}alder{\'o}n
  problem with a single measurement.
\newblock {\em Journal of Functional Analysis}, page 108505, 2020.

\bibitem[GS17]{garde2017convergence}
Henrik Garde and Stratos Staboulis.
\newblock Convergence and regularization for monotonicity-based shape
  reconstruction in electrical impedance tomography.
\newblock {\em Numerische Mathematik}, 135(4):1221--1251, 2017.

\bibitem[GS19]{garde2019regularized}
Henrik Garde and Stratos Staboulis.
\newblock The regularized monotonicity method: Detecting irregular indefinite
  inclusions.
\newblock {\em Inverse Probl. Imaging}, 13(1):93--116, 2019.

\bibitem[GSU20]{ghosh2016calder}
Tuhin Ghosh, Mikko Salo, and Gunther Uhlmann.
\newblock The {C}alder{\'o}n problem for the fractional {S}chr{\"o}dinger
  equation.
\newblock {\em Analysis \& PDE}, 13(2):455--475, 2020.

\bibitem[Har09]{harrach2009uniqueness}
Bastian Harrach.
\newblock On uniqueness in diffuse optical tomography.
\newblock {\em Inverse Problems}, 25:055010 (14pp), 2009.

\bibitem[Har12]{harrach2012simultaneous}
Bastian Harrach.
\newblock Simultaneous determination of the diffusion and absorption
  coefficient from boundary data.
\newblock {\em Inverse Probl. Imaging}, 6(4):663--679, 2012.

\bibitem[Har19]{harrach2019uniqueness}
Bastian Harrach.
\newblock Uniqueness and {L}ipschitz stability in electrical impedance
  tomography with finitely many electrodes.
\newblock {\em Inverse Problems}, 35(2):024005, 2019.

\bibitem[HL19]{harrach2017nonlocal-monotonicity}
Bastian Harrach and Yi-Hsuan Lin.
\newblock Monotonicity-based inversion of the fractional {S}chr\"odinger
  equation {I}. {P}ositive potentials.
\newblock {\em SIAM Journal on Mathematical Analysis}, 51(4):3092--3111, 2019.

\bibitem[HL20]{harrach2020monotonicity}
Bastian Harrach and Yi-Hsuan Lin.
\newblock Monotonicity-based inversion of the fractional {S}ch\"odinger
  equation {II}. {G}eneral potentials and stability.
\newblock {\em SIAM Journal on Mathematical Analysis}, 52(1):402--436, 2020.

\bibitem[HLL18]{harrach2018localizing}
Bastian Harrach, Yi-Hsuan Lin, and Hongyu Liu.
\newblock On localizing and concentrating electromagnetic fields.
\newblock {\em SIAM J. Appl. Math}, 78(5):2558--2574, 2018.

\bibitem[HLU15]{harrach2015combining}
Bastian Harrach, Eunjung Lee, and Marcel Ullrich.
\newblock Combining frequency-difference and ultrasound modulated electrical
  impedance tomography.
\newblock {\em Inverse Problems}, 31(9):095003, 2015.

\bibitem[HM16]{harrach2016enhancing}
Bastian Harrach and Mach~Nguyet Minh.
\newblock Enhancing residual-based techniques with shape reconstruction
  features in electrical impedance tomography.
\newblock {\em Inverse Problems}, 32(12):125002, 2016.

\bibitem[HM18]{harrach2018monotonicity}
Bastian Harrach and Mach~Nguyet Minh.
\newblock Monotonicity-based regularization for phantom experiment data in
  electrical impedance tomography.
\newblock In {\em New Trends in Parameter Identification for Mathematical
  Models}, pages 107--120. Springer, 2018.

\bibitem[HM19]{HM2019global_stability}
Bastian Harrach and Houcine Meftahi.
\newblock Global uniqueness and {L}ipschitz-stability for the inverse {R}obin
  transmission problem.
\newblock {\em SIAM J. Appl. Math.}, 79(2):525--550, 2019.

\bibitem[Hor85]{hormander1983analysis}
Lars Hormander.
\newblock {\em The Analysis of Linear Partial Differential Operators. {I-IV}}.
\newblock 1983-1985.

\bibitem[HPS19a]{HPS2019dimension}
Bastian Harrach, Valter Pohjola, and Mikko Salo.
\newblock Dimension bounds in monotonicity methods for the {H}elmholtz
  equation.
\newblock {\em SIAM Journal on Mathematical Analysis}, 51(4):2995--3019, 2019.

\bibitem[HPS19b]{HPSmonotonicity}
Bastian Harrach, Valter Pohjola, and Mikko Salo.
\newblock Monotonicity and local uniqueness for the {H}elmholtz equation.
\newblock {\em Analysis \& PDE}, 12(7):1741--1771, 2019.

\bibitem[HS10]{harrach2010exact}
Bastian Harrach and Jin~Keun Seo.
\newblock Exact shape-reconstruction by one-step linearization in electrical
  impedance tomography.
\newblock {\em SIAM Journal on Mathematical Analysis}, 42(4):1505--1518, 2010.

\bibitem[HU13]{harrach2013monotonicity}
Bastian Harrach and Marcel Ullrich.
\newblock Monotonicity-based shape reconstruction in electrical impedance
  tomography.
\newblock {\em SIAM Journal on Mathematical Analysis}, 45(6):3382--3403, 2013.

\bibitem[HU15]{harrach2015resolution}
Bastian Harrach and Marcel Ullrich.
\newblock Resolution guarantees in electrical impedance tomography.
\newblock {\em IEEE Trans. Med. Imaging}, 34:1513--1521, 2015.

\bibitem[HU17]{harrach2017local}
Bastian Harrach and Marcel Ullrich.
\newblock Local uniqueness for an inverse boundary value problem with partial
  data.
\newblock {\em Proceedings of the American Mathematical Society},
  145(3):1087--1095, 2017.

\bibitem[KU19]{KU2019partial}
Katya Krupchyk and Gunther Uhlmann.
\newblock Partial data inverse problems for semilinear elliptic equations with
  gradient nonlinearities.
\newblock {\em arXiv:1909.08122v1}, 2019.

\bibitem[KU20]{KU2019remark}
Katya Krupchyk and Gunther Uhlmann.
\newblock A remark on partial data inverse problems for semilinear elliptic
  equations.
\newblock {\em Proc. Amer. Math. Soc.}, 148:681--685, 2020.

\bibitem[KZ16]{Kay}
Kay Kirkpatrick and Yanzhi Zhang.
\newblock Fractional {S}chr\"odinger dynamics and decoherence.
\newblock {\em Physica D: Nonlinear Phenomena}, 332(14):41--54, 2016.

\bibitem[LL19]{lai2019global}
Ru-Yu Lai and Yi-Hsuan Lin.
\newblock Global uniqueness for the fractional semilinear {S}chr{\"o}dinger
  equation.
\newblock {\em Proc. Amer. Math. Soc.}, 147(3):1189--1199, 2019.

\bibitem[LL20]{LL2020inverse}
Ru-Yu Lai and Yi-Hsuan Lin.
\newblock Inverse problems for fractional semilinear elliptic equations.
\newblock {\em arXiv preprint arXiv:2004.00549}, 2020.

\bibitem[LLLS20a]{LLLS2019nonlinear}
Matti Lassas, Tony Liimatainen, Yi-Hsuan Lin, and Mikko Salo.
\newblock Inverse problems for elliptic equations with power type
  nonlinearities.
\newblock {\em Journal de Math{\'e}matiques Pures et Appliqu{\'e}es, in press},
  2020.

\bibitem[LLLS20b]{LLLS2019partial}
Matti Lassas, Tony Liimatainen, Yi-Hsuan Lin, and Mikko Salo.
\newblock Partial data inverse problems and simultaneous recovery of boundary
  and coefficients for semilinear elliptic equations.
\newblock {\em Revista Matematica Iberoamericana, accepted for publication},
  2020.

\bibitem[LLR20]{LLR2019calder}
Ru-Yu Lai, Yi-Hsuan Lin, and Angkana R{\"u}land.
\newblock The {C}alder\'on problem for a space-time fractional parabolic
  equation.
\newblock {\em SIAM Journal on Mathematical Analysis, accepted for
  publication}, 2020.

\bibitem[McL00]{mclean2000strongly}
William Charles~Hector McLean.
\newblock {\em Strongly elliptic systems and boundary integral equations}.
\newblock Cambridge University Press, 2000.

\bibitem[MVVT16]{maffucci2016novel}
Antonio Maffucci, Antonio Vento, Salvatore Ventre, and Antonello Tamburrino.
\newblock A novel technique for evaluating the effective permittivity of
  inhomogeneous interconnects based on the monotonicity property.
\newblock {\em IEEE Transactions on Components, Packaging and Manufacturing
  Technology}, 6(9):1417--1427, 2016.

\bibitem[RO16]{ros2015nonlocal}
Xavier Ros-Oton.
\newblock Nonlocal elliptic equations in bounded domains: a survey.
\newblock {\em Publicacions Matem\`{e}tiques.}, 60(1):3--26, 2016.

\bibitem[ROS14a]{ros2014dirichlet}
Xavier Ros-Oton and Joaquim Serra.
\newblock The {D}irichlet problem for the fractional {L}aplacian: regularity up
  to the boundary.
\newblock {\em Journal de Math{\'e}matiques Pures et Appliqu{\'e}es},
  101(3):275--302, 2014.

\bibitem[ROS14b]{RS2014extremal}
Xavier Ros-Oton and Joaquim Serra.
\newblock The extremal solution for the fractional {L}aplacian.
\newblock {\em Calculus of variations and partial differential equations},
  50(3-4):723--750, 2014.

\bibitem[RS19]{RS2018lipschitz}
Angkana R{\"u}land and Eva Sincich.
\newblock Lipschitz stability for the finite dimensional fractional
  {C}alder\'on problem with finite {C}auchy data.
\newblock {\em Inverse Problems and Imaging}, 13(5), 2019.

\bibitem[RS20]{RS17}
Angkana R{\"u}land and Mikko Salo.
\newblock The fractional {C}alder{\'o}n problem: low regularity and stability.
\newblock {\em Nonlinear Analysis}, 193:111529, 2020.

\bibitem[Sin07]{sincich2007lipschitz}
Eva Sincich.
\newblock Lipschitz stability for the inverse {R}obin problem.
\newblock {\em Inverse Problems}, 23(3):1311, 2007.

\bibitem[SKJ{\etalchar{+}}19]{seo2018learning}
Jin~Keun Seo, Kang~Cheol Kim, Ariungerel Jargal, Kyounghun Lee, and Bastian
  Harrach.
\newblock A learning-based method for solving ill-posed nonlinear inverse
  problems: a simulation study of lung {EIT}.
\newblock {\em SIAM Journal on Imaging Sciences}, 12(3):1275--1295, 2019.

\bibitem[Ste16]{stein2016singular}
Elias~M Stein.
\newblock {\em Singular integrals and differentiability properties of functions
  (PMS-30)}, volume~30.
\newblock Princeton university press, 2016.

\bibitem[SUG{\etalchar{+}}17]{su2017monotonicity}
Zhiyi Su, Lalita Udpa, Gaspare Giovinco, Salvatore Ventre, and Antonello
  Tamburrino.
\newblock Monotonicity principle in pulsed eddy current testing and its
  application to defect sizing.
\newblock In {\em Applied Computational Electromagnetics Society
  Symposium-Italy (ACES), 2017 International}, pages 1--2. IEEE, 2017.

\bibitem[TR02]{tamburrino2002new}
Antonello Tamburrino and Guglielmo Rubinacci.
\newblock A new non-iterative inversion method for electrical resistance
  tomography.
\newblock {\em Inverse Problems}, 18(6):1809, 2002.

\bibitem[TSV{\etalchar{+}}16]{tamburrino2016monotonicity}
Antonello Tamburrino, Zhiyi Sua, Salvatore Ventre, Lalita Udpa, and Satish~S
  Udpa.
\newblock Monotonicity based imang method in time domain eddy current testing.
\newblock {\em Electromagnetic Nondestructive Evaluation (XIX)}, 41:1, 2016.

\bibitem[UB13]{Uzar}
Neslihan Uzar and Sedat Ballikaya.
\newblock Investigation of classical and fractional {B}ose-{E}instein
  condensation for harmonic potential.
\newblock {\em Physica A: Statistical Mechanics and its Applications},
  392(8):1733--1741, 2013.

\bibitem[VMC{\etalchar{+}}17]{ventre2017design}
Salvatore Ventre, Antonio Maffucci, Fran{\c{c}}ois Caire, Nechtan Le~Lostec,
  Antea Perrotta, Guglielmo Rubinacci, Bernard Sartre, Antonio Vento, and
  Antonello Tamburrino.
\newblock Design of a real-time eddy current tomography system.
\newblock {\em IEEE Transactions on Magnetics}, 53(3):1--8, 2017.

\bibitem[ZHS18]{zhou2018monotonicity}
Liangdong Zhou, Bastian Harrach, and Jin~Keun Seo.
\newblock Monotonicity-based electrical impedance tomography for lung imaging.
\newblock {\em Inverse Problems}, 34(4):045005, 2018.

\end{thebibliography}

\end{document}